\documentclass[11pt,a4paper]{article}
\usepackage[titletoc,title]{appendix}
\usepackage{amsmath,amsfonts,amssymb,bm,hyperref,amsthm,mathrsfs}
\numberwithin{equation}{section}
\newtheorem{theorem}{Theorem}[section]
\newtheorem{lemma}[theorem]{Lemma}

\theoremstyle{definition}

\newtheorem{remark}[theorem]{Remark}

\usepackage[top=1in, bottom=1in, left=1in, right=1in]{geometry}
\usepackage{fancyhdr}
\begin{document}

\title{Spectral analysis of the Dirac operator on a 3-sphere}
\author{Yan-Long Fang\thanks{YF:
Department of Mathematics,
University College London,
Gower Street,
London WC1E~6BT,
UK;
yan.fang.12@ucl.ac.uk.
}
\and
Michael Levitin\thanks{ML:
Department of Mathematics and Statistics,
University of Reading,
Whiteknights,
PO Box 220,
Reading RG6~6AX,
UK;
M.Levitin@reading.ac.uk,
\url{http://www.personal.reading.ac.uk/\~ny901965/}.
}
\and
Dmitri Vassiliev\thanks{DV:
Department of Mathematics,
University College London,
Gower Street,
London WC1E~6BT,
UK;
D.Vassiliev@ucl.ac.uk,
\url{http://www.homepages.ucl.ac.uk/\~ucahdva/};
DV was supported by EPSRC grant EP/M000079/1.
}}

\renewcommand\footnotemark{}



\maketitle
\begin{abstract}
We study the (massless) Dirac operator on a 3-sphere equipped with
Riemannian metric. For the standard metric the spectrum is known.
In particular, the eigenvalues closest to zero are the two double
eigenvalues $+3/2$ and $-3/2$.
Our aim is to analyse the behaviour of eigenvalues when
the metric is perturbed in an arbitrary smooth fashion from the standard one.
We derive explicit perturbation formulae for the two eigenvalues closest to zero,
taking account of the second variations.
Note that these eigenvalues remain double eigenvalues under perturbations
of the metric: they cannot split because of a particular symmetry
of the Dirac operator in dimension three
(it commutes with the antilinear operator of charge conjugation).
Our perturbation formulae show that in the first approximation our two eigenvalues
maintain symmetry about zero and are completely determined
by the increment of Riemannian volume.
Spectral asymmetry is observed only in the second approximation of the perturbation
process.
As an example we consider a special family of metrics,
the so-called generalized Berger spheres,
for which the eigenvalues can be evaluated explicitly.

\


\

{\bf MSC classes: }35Q41; 35P15; 58J50; 53C25.

\end{abstract}

\newpage
\tableofcontents

\section{Introduction}
\label{Introduction}

In this paper we study the spectrum of the (massless)
Dirac operator on a 3-sphere,
$\mathbb{S}^3$, equipped with Riemannian metric.

By $y^\alpha$, $\alpha=1,2,3$, we denote local coordinates.
We specify an orientation, see Appendix~\ref{Orientation},
and use only local coordinates with positive orientation.

We will use the following conventions.
Sometimes it will be convenient for us to view
the 3-sphere as the hypersurface
\eqref{definition of 3-sphere} in Euclidean space
$\mathbb{R}^4$, in which case Cartesian coordinates
in $\mathbb{R}^4$ will be denoted by
$\mathbf{x}^{\boldsymbol\alpha}$, $\boldsymbol\alpha=1,2,3,4$.
Hereinafter we will use bold script for 4-dimensional objects and normal script for
3-dimensional objects.
We will use Latin letters for \emph{anholonomic (frame)} indices
and Greek letters for \emph{holonomic (tensor)} indices.
We will use the convention of summation
over repeated indices; this will apply both to
Greek and to Latin indices.
Also, we will heavily use the analytic concepts
of principal and subprincipal symbol of a differential operator;
see definitions in
\cite[subsection 2.1.3]{mybook} for the case of a scalar operator
acting on a single half-density
and, more relevantly,
\cite[Section 1]{israel}
and
\cite[Appendix A]{action}
for the case of a matrix operator acting on a column of half-densities.

We equip $\mathbb{S}^3$ with a Riemannian metric tensor $g_{\alpha\beta}(y)$,
$\alpha,\beta=1,2,3$, and study the corresponding
(massless) Dirac operator $W$.
The Dirac operator on a 3-manifold is a particular operator that can be represented as a first order
elliptic linear differential operator acting on 2-columns of complex-valued scalar fields
(components of a Weyl spinor).
It is written down explicitly in Appendix~\ref{The Dirac operator};
note that the definition depends on the choice of orientation,
see formula \eqref{Analytic definition condition 4}
or formula \eqref{orientation of frame}.
It is known that the Dirac operator $W$ is a self-adjoint operator in
$L^2(\mathbb{S}^3;\mathbb{C}^2)$ whose domain
is the Sobolev space $H^1(\mathbb{S}^3;\mathbb{C}^2)$,
and that the spectrum of $W$ is discrete, accumulating to
$+\infty$ and to~$-\infty$.
Here the inner product in $L^2(\mathbb{S}^3;\mathbb{C}^2)$ is defined as
\begin{equation}
\label{inner product on 2-columns of scalar fields}
\langle v,w\rangle:=\int_{\mathbb{S}^3}
\bigl(w^*v\,\sqrt{\det g_{\alpha\beta}}\,\bigr)\,dy\,,
\end{equation}
where the star stands for Hermitian conjugation
and $dy=dy^1dy^2dy^3$.
Furthermore, it is known that all eigenvalues have even multiplicity
because the linear Dirac operator commutes with the antilinear operator
of charge conjugation
\begin{equation*}
v=
\begin{pmatrix}
v_1\\
v_2
\end{pmatrix}
\mapsto
\begin{pmatrix}
-\overline{v_2}\\
\overline{v_1}
\end{pmatrix}
=:\mathrm{C}(v),
\end{equation*}
see \cite[Appendix A]{jst_part_b} for details.

A detailed analysis of different definitions of the massless Dirac operator
and their equivalence was performed in \cite{FangPhD}.
Let us emphasise that the underlying reason why we can use three equivalent definitions
presented in subsections
\ref{Classical geometric definition},
\ref{Definition via frames} and
\ref{Analytic definition}
is that our manifold is 3-dimensional.
One can define the Dirac operator via frames (subsection~\ref{Definition via frames})
or the covariant symbol (sub\-section~\ref{Analytic definition}) in dimension 3 (Riemannian signature)
or in dimension $3+1$ (Lorentzian signature), see \cite{nongeometric} for details,
but for higher dimensions these approaches do not seem to work.

The Dirac operator is uniquely defined by the metric modulo the gauge transformation
\begin{equation}
\label{transformation of the Dirac operator under change of frame}
W\mapsto R^*WR,
\end{equation}
where
\begin{equation}
\label{special unitary matrix-function}
R:\mathbb{S}^3\to\mathrm{SU}(2)
\end{equation}
is an arbitrary smooth special unitary matrix-function;
see also subsection~\ref{Spin structure}
for a discussion of spin structure.
Obviously, the transformation
\eqref{transformation of the Dirac operator under change of frame},
\eqref{special unitary matrix-function}
does not affect the spectrum.

The Dirac operator $W$ describes the massless neutrino.
We are looking at a single neutrino living in a closed 3-dimensional
Riemannian universe. The eigenvalues are the energy levels of the particle.
The tradition is to associate positive eigenvalues with the energy levels of the neutrino
and negative eigenvalues with the energy levels of the antineutrino.
In theoretical physics literature the (massless) Dirac operator
is often referred to as the Weyl operator
which explains our notation.

In what follows we will mostly conduct a perturbation analysis:
starting from the standard metric on the 3-sphere and perturbing this
metric in an arbitrary fashion, we will write down explicit asymptotic
formulae for the two eigenvalues closest to zero.
At a more abstract level the behaviour of eigenvalues of the massless Dirac operator
under perturbations of the metric was studied in \cite{bourguignon}.
Our analysis is very specific in that
\begin{itemize}
	\item the dimension is three,
	\item our manifold is a topological sphere,
	\item the unperturbed metric is the standard metric for the sphere,
	\item we study only the two eigenvalues closest to zero and
	\item we calculate the second variations for the perturbed eigenvalues
	(Theorem \ref{theorem 2}),
	which is necessary for the observation of spectral asymmetry.
\end{itemize}
We are motivated by the desire to analyse, in detail and explicitly, the problem in a physically meaningful setting.
Namely, we are trying to understand the difference between neutrinos and antineutrinos in curved space.
For other related problems in mathematics and physics literature see, e.g., \cite{dowker1,dowker2}.

Finally, in
Section~\ref{Generalized Berger spheres} and
Appendix~\ref{Explicit formulae for some eigenvalues of Dirac operator on generalized Berger spheres}
we study, non-perturbatively, \emph{all} the eigenvalues for a particular
three-parameter family of metrics known as generalized Berger spheres,
cf.~\cite{godbout}. This extends previous results of \cite{bar_1992,bar_1996a}.

\section{Main results}
\label{Main results}

The standard metric $(g_0)_{\alpha\beta}(y)$ on $\mathbb{S}^3$
is obtained by restricting
the Euclidean metric from $\mathbb{R}^4$ to $\mathbb{S}^3$.
For the standard metric the spectrum of the (massless) Dirac operator on
$\mathbb{S}^3$ has been computed by different authors using different methods
\cite{sulanke,trautman,bar_1996,bar_2000} and reads as follows:
the eigenvalues are
\begin{equation*}
\pm\left(k+\frac12\right),
\qquad k=1,2,\ldots,
\end{equation*}
with multiplicity
$k(k+1)$.

We now perturb the metric, i.e.~consider a metric
$g_{\alpha\beta}=g_{\alpha\beta}(y;\epsilon)$ whose components are smooth
functions of local coordinates $y^\alpha$,
$\alpha=1,2,3$, and small real parameter $\epsilon$, and which
turns into the standard metric
for $\epsilon=0$:
\[
g_{\alpha\beta}(y;0)=(g_0)_{\alpha\beta}(y).
\]

Let $\lambda_+(\epsilon)$
and $\lambda_-(\epsilon)$ be the lowest, in terms of absolute value, positive
and negative eigenvalues of the Dirac operator $W(\epsilon)$.
Our aim is to derive the asymptotic expansions
\begin{equation}
\label{asymptotic expansion for smallest eigenvalues}
\lambda_\pm(\epsilon)=\pm\frac32+\lambda_\pm^{(1)}\epsilon
+\lambda_\pm^{(2)}\epsilon^2+O(\epsilon^3)
\qquad\text{as}\qquad\epsilon\to0.
\end{equation}
Note that $\lambda_\pm(\epsilon)$ are double eigenvalues which cannot
split because eigenvalues of the Dirac operator have even multiplicity.
Note also that the arguments presented in \cite{torus}
apply to any double eigenvalue of the Dirac operator
on any closed orientable Riemannian 3-manifold,
so we know a priori that $\lambda_\pm(\epsilon)$ admit the asymptotic
expansions \eqref{asymptotic expansion for smallest eigenvalues}.
The issue at hand is the evaluation of the asymptotic coefficients
$\lambda_\pm^{(1)}$ and $\lambda_\pm^{(2)}$.

Let
\begin{equation}
\label{definition of volume}
V(\epsilon):=\int_{\mathbb{S}^3}
\rho(\epsilon)\,dy\,,
\end{equation}
$dy=dy^1dy^2dy^3$,
be the Riemannian volume of our manifold.
Here 
\begin{equation*}
\rho(\epsilon)=\rho(y;\epsilon):=\sqrt{\det g_{\mu\nu}(y;\epsilon)}
\end{equation*}
is the Riemannian density for the perturbed metric.

Then
\begin{equation}
\label{asymptotic expansion for volume}
V(\epsilon)=V^{(0)}+V^{(1)}\epsilon
+O(\epsilon^2)
\qquad\text{as}\qquad\epsilon\to0,
\end{equation}
where
\begin{equation}
\label{formula for V0}
V^{(0)}=\int_{\mathbb{S}^3}
\rho_0\,dy\,=\,2\pi^2
\end{equation}
is the volume of the unperturbed 3-sphere,
\begin{equation*}
\rho_0=\rho_0(y):=\sqrt{\det(g_0)_{\mu\nu}(y)}
\end{equation*}
is the standard Riemannian density on the 3-sphere,
\begin{equation}
\label{formula for V1}
V^{(1)}=\frac12\int_{\mathbb{S}^3}
h_{\alpha\beta}\,(g_0)^{\alpha\beta}\,\rho_0\,dy
\end{equation}
and
\begin{equation}
\label{definition of tensor h}
h_{\alpha\beta}:=
\left.
\frac{\partial g_{\alpha\beta}}{\partial\epsilon}
\right|_{\epsilon=0}\,.
\end{equation}

\begin{theorem}
\label{theorem 1}
We have
\begin{equation}
\label{formula for lambda1}
\lambda_\pm^{(1)}=\mp\frac1{4\pi^2}V^{(1)}.
\end{equation}
\end{theorem}

We see that the dependence of the two lowest eigenvalues, $\lambda_\pm(\epsilon)$,
on the small parameter $\epsilon$ is, in the first approximation, very simple:
it is determined by the change of volume only.
As expected, an increase of the volume of the resonator (volume of our Riemannian manifold)
leads to a decrease of the two lowest natural frequencies
(absolute values of the two lowest eigenvalues).
Furthermore, formulae
\eqref{asymptotic expansion for volume},
\eqref{formula for V0}
and
\eqref{formula for lambda1}
imply
\begin{equation}
\label{formula for lambda1 with one third}
\frac{\lambda_\pm^{(1)}}{\lambda_\pm^{(0)}}
=-\frac13
\frac{V^{(1)}}{V^{(0)}}\,,
\end{equation}
where by $\lambda_\pm^{(0)}=\pm\frac32$
we denoted the unperturbed values of the two lowest eigenvalues.
Now put $\ell(\epsilon):=(V(\epsilon))^{1/3}
=\ell^{(0)}\left(1+\frac13
\frac{V^{(1)}}{V^{(0)}}\epsilon+O(\epsilon^2)\right)$,
where $\ell^{(0)}=\ell(0)=(2\pi^2)^{1/3}$.
The quantity $\ell(\epsilon)$ can be interpreted as the characteristic length
of our Riemannian manifold.
It is easy to see that formula \eqref{formula for lambda1 with one third}
is equivalent to the statement
\begin{equation*}
\lambda_\pm(\epsilon)=
\frac{\lambda_\pm^{(0)}}{\ell(\epsilon)}+O(\epsilon^2)
\qquad\text{as}\qquad\epsilon\to0,
\end{equation*}
which shows that in the first approximation the two lowest eigenvalues are inversely proportional to the characteristic length of our Riemannian manifold.

An important topic in the spectral theory of first order elliptic systems is the
issue of spectral asymmetry
\cite{atiyah_short_paper,atiyah_part_1,atiyah_part_2,atiyah_part_3,torus},
i.e.~asymmetry of the spectrum about zero.
From a physics perspective spectral asymmetry describes the difference between
a particle (in our case massless neutrino)
and
an antiparticle (in our case massless antineutrino).
Formulae
\eqref{asymptotic expansion for smallest eigenvalues}
and
\eqref{formula for lambda1}
imply
\begin{equation*}
\lambda_+(\epsilon)+\lambda_-(\epsilon)
=\bigl(\lambda_+^{(2)}+\lambda_-^{(2)}\bigr)\epsilon^2+O(\epsilon^3)
\qquad\text{as}\qquad\epsilon\to0,
\end{equation*}
which means that there is no spectral asymmetry in the first approximation in~$\epsilon$
but there may be spectral asymmetry in terms quadratic in $\epsilon$.

We will now evaluate the asymptotic coefficients $\lambda_\pm^{(2)}$.
We will do this under the simplifying assumption that the Riemannian
density does not depend on $\epsilon$:
\begin{equation}
\label{shear}
\sqrt{\det g_{\mu\nu}(y;\epsilon)}
=
\sqrt{\det(g_0)_{\mu\nu}(y)}\,,
\end{equation}
so that $V^{(1)}=0$.
In mechanics such a deformation is called \emph{shear}.
Then
Theorem~\ref{theorem 1}
implies $\lambda_\pm^{(1)}=0$,
so formula \eqref{asymptotic expansion for smallest eigenvalues} now reads
\begin{equation}
\label{asymptotic expansion for smallest eigenvalues shear}
\lambda_\pm(\epsilon)=\pm\frac32
+\lambda_\pm^{(2)}\epsilon^2+O(\epsilon^3)
\qquad\text{as}\qquad\epsilon\to0.
\end{equation}

In order to evaluate the asymptotic coefficients $\lambda_\pm^{(2)}$
we need to introduce triples of special vector fields
$(K_\pm)_j\,$, $j=1,2,3$.
For their definitions and properties see
Appendix~\ref{Special vector fields on the 3-sphere}.
Here we mention only that these are
triples of orthonormal Killing fields with respect to the standard (unperturbed) metric.

Put
\begin{equation}
\label{definition of tensor h frame}
(h_\pm)_{jk}:=
h_{\alpha\beta}\,(K_\pm)_j{}^\alpha\,(K_\pm)_k{}^\beta,
\end{equation}
where $h_{\alpha\beta}$ is the real symmetric tensor from
\eqref{definition of tensor h}.
Note that the elements of the $3\times3$ real symmetric matrix-function
$(h_\pm)_{jk}(y)$ are scalars,
i.e.~they do not change under changes of local coordinates $y$.
Further on we sometimes
raise and lower frame indices (see subsection \ref{Definition via frames})
and we do this using the Euclidean metric.
This means, in particular,
that raising a frame index
in $(h_\pm)_{jk}$
does not change anything. 

Put also
\begin{equation}
\label{definition of operator Lj}
(L_\pm)_j:=
(K_\pm)_j{}^\alpha\,\frac\partial{\partial y^\alpha}\,,
\qquad j=1,2,3.
\end{equation}
The operators \eqref{definition of operator Lj} are first order linear
differential operators acting on scalar fields over $\mathbb{S}^3$.
The fact that our $(K_\pm)_j$ are Killing vector fields implies
that the operators \eqref{definition of operator Lj} are formally
anti-self-adjoint with respect to the standard inner product
on scalar fields over $\mathbb{S}^3$.
It is also easy to see that
our operators $(L_\pm)_j$, $j=1,2,3$, satisfy the commutator identities
\begin{equation}
\label{commutator identities}
[(L_\pm)_j,(L_\pm)_k]=\mp2\varepsilon_{jkl}(L_\pm)_l\,,
\end{equation}
where $\varepsilon_{jkl}$
is the totally antisymmetric quantity, $\varepsilon_{123}:=+1$.

Let $\Delta$ be the Laplacian on scalar fields over $\mathbb{S}^3$
with standard (unperturbed) metric.
Our $\Delta$ is a
nonpositive operator, so our definition agrees with the one from basic calculus.
By $(-\Delta)^{-1}$ we shall denote the pseudoinverse of the non-negative
differential operator $-\Delta$, see
Appendix \ref{The scalar Laplacian and its pseudoinverse}
for explicit definition.
Obviously, $(-\Delta)^{-1}$ is a classical pseudodifferential operator
of order minus two.

\begin{theorem}
\label{theorem 2}
Under the assumption \eqref{shear}
we have
\begin{equation}
\label{formula for lambda2}
\lambda_\pm^{(2)}=\frac1{2\pi^2}\int_{\mathbb{S}^3}
P_\pm\,\rho_0\,dy\,,
\end{equation}
where
\begin{equation}
\begin{split}
\label{formula for lambda2 quadratic form}
P_\pm&=
\pm\frac14(h_\pm)_{jk}(h_\pm)_{jk}
\\
&\quad-\frac1{16}\varepsilon_{qks}(h_\pm)_{jq}\left[(L_\pm)_s(h_\pm)_{jk}\right]
\\
&\quad\pm\frac18(h_\pm)_{ks}\left[(-\Delta)^{-1}(L_\pm)_{s}(L_\pm)_j(h_\pm)_{jk}\right]
\\
&\quad-\frac1{16}\varepsilon_{qks}(h_\pm)_{rq}
\left[(-\Delta)^{-1}(L_\pm)_{r}(L_\pm)_s(L_\pm)_j(h_\pm)_{jk}\right].
\end{split}
\end{equation}
\end{theorem}

Theorem \ref{theorem 2} warrants the following remarks.

\begin{remark}
\label{remarks for theorem 2}
\begin{enumerate}
\item[(a)]
We chose the factor $\frac1{2\pi^2}$ in the RHS of
\eqref{formula for lambda2} based on the observation that
the volume of the unperturbed 3-sphere is $2\pi^2$,
see formula \eqref{formula for V0}.
This will simplify the comparison with the appropriate formulae
previously derived for the 3-torus, see item (f) below,
and it will also simplify calculations that will be carried out
in subsection~\ref{Testing Theorem 2 on generalized Berger spheres}.
\item[(b)]
The terms in the RHS of \eqref{formula for lambda2 quadratic form}
are written in such an order that each
subsequent term has an extra appearance of a first
order differential operator $L_\pm\,$.
\item[(c)]
The operators
$(-\Delta)^{-1}(L_\pm)_s(L_\pm)_j$
and
$(-\Delta)^{-1}(L_\pm)_r(L_\pm)_s(L_\pm)_j$
appearing in the last two terms in the
RHS of \eqref{formula for lambda2 quadratic form}
are pseudodifferential operators of order 0 and 1 respectively.
\item[(d)]
The fact that the operators $(L_\pm)_j$, $j=1,2,3$, are formally
anti-self-adjoint with respect to the standard inner product
on $\mathbb{S}^3$ implies that for any scalar field
$f:\mathbb{S}^3\to\mathbb{C}$ we have
\begin{equation}
\label{L maps to orthogonal complement of kernel of Laplacian}
\int_{\mathbb{S}^3}\left[(L_\pm)_jf\right]\rho_0\,dy=0\,.
\end{equation}
Formula \eqref{L maps to orthogonal complement of kernel of Laplacian}
implies that in
the last two terms in the
RHS of \eqref{formula for lambda2 quadratic form}
the operator $(-\Delta)^{-1}$ acts on functions
from $(\operatorname{Ker}\Delta)^\perp$.
\item[(e)]
The operators $(L_\pm)_j$ commute with the scalar Laplacian, hence,
they also commute with $(-\Delta)^{-1}$.
Therefore,
the last two terms in the
RHS of \eqref{formula for lambda2 quadratic form}
can be written in a number of equivalent ways.
\item[(f)]
The second and fourth terms in the
RHS of \eqref{formula for lambda2 quadratic form}
have a structure similar to that of formula (2.5)
from \cite{torus}. In fact, if one adjusts notation
to agree with that of \cite{torus}, then it turns out that
the second and fourth terms in the
RHS of \eqref{formula for lambda2 quadratic form}
are, in effect, an equivalent way of writing formula (2.5)
from \cite{torus}. See Appendix \ref{Comparison with the 3-torus} for more details.
\item[(g)]
The first and third terms in the
RHS of \eqref{formula for lambda2 quadratic form}
do not have an analogue for the case of the 3-torus \cite{torus}.
Their appearance is due to the curvature of the 3-sphere.
\item[(h)]
The first term in the
RHS of \eqref{formula for lambda2 quadratic form}
can be rewritten as
\begin{equation}
\label{elasticity1}
\pm\frac14
h_{\mu\nu}\,h_{\sigma\tau}\,(g_0)^{\mu\sigma}\,(g_0)^{\nu\tau}\,,
\end{equation}
which means that this term does not feel the Killing vector fields
$(K_\pm)_j\,$, $j=1,2,3$, and, hence, does not contribute to spectral asymmetry.
Put
\begin{equation*}
\widetilde h_{\mu\nu}:=
h_{\mu\nu}-\frac13\,\delta_{\mu\nu}\,h_{\sigma\tau}\,(g_0)^{\sigma\tau}\,,
\end{equation*}
which is the part of the deformation tensor $h_{\mu\nu}$ describing shear
(deformation preserving Riemannian density).
Formula \eqref{shear} implies $\,h_{\sigma\tau}\,(g_0)^{\sigma\tau}=0\,$,
so in our case $\widetilde h_{\mu\nu}=h_{\mu\nu}$ and the expression
\eqref{elasticity1} takes the form
\begin{equation*}
\pm\frac14
\widetilde h_{\mu\nu}\,\widetilde h_{\sigma\tau}\,
(g_0)^{\mu\sigma}\,(g_0)^{\nu\tau}\,.
\end{equation*}
Such an expression describes the elastic potential energy
generated by shear, see formula (4.3) in \cite{LL_elasticity}.
\item[(i)]
For a generic perturbation of the metric we expect
\begin{equation}
\label{spectral asymmetry in quadratic terms}
\lambda_+^{(2)}+\lambda_-^{(2)}\ne0,
\end{equation}
which means that we expect spectral asymmetry in terms quadratic in $\epsilon$.
An example illustrating the inequality \eqref{spectral asymmetry in quadratic terms}
will be provided in subsection
\ref{Testing Theorem 2 on generalized Berger spheres}:
see formulae
\eqref{lowest positive eigenvalue for generalized Berger metric expansion order 2}
and
\eqref{lowest negative eigenvalue for generalized Berger metric expansion order 2}.
\item[(j)]
Let us expand the metric tensor in powers of the small parameter $\epsilon$
up to quadratic terms:
\[
g_{\alpha\beta}(y;\epsilon)=(g_0)_{\alpha\beta}(y)
+h_{\alpha\beta}(y)\,\epsilon
+k_{\alpha\beta}(y)\,\epsilon^2+O(\epsilon^3)
\qquad\text{as}\qquad\epsilon\to0.
\]
Here the tensor $h_{\alpha\beta}$ is defined by
\eqref{definition of tensor h} whereas
$k_{\alpha\beta}:=
\frac12
\left.
\frac{\partial^2 g_{\alpha\beta}}{\partial\epsilon^2}
\right|_{\epsilon=0}\,$.
One would expect the coefficients $\lambda_\pm^{(2)}$ in the asymptotic expansions
\eqref{asymptotic expansion for smallest eigenvalues shear}
of the
lowest eigenvalues to depend on the tensor $k_{\alpha\beta}\,$, but
Theorem~\ref{theorem 2}
tells us that it is not the case.
Here a rough explanation is that the only way the tensor $k_{\alpha\beta}$ can
appear in the formulae for $\lambda_\pm^{(2)}$ is linearly, however, condition
\eqref{shear} ensures that the linear terms in the map
\[
\text{perturbation of metric}
\quad\to\quad
\text{perturbation of lowest eigenvalues}
\]
vanish.
\end{enumerate}
\end{remark}

\section{Preparatory material}
\label{Preparatory material}

In this section we present auxiliary results which will be used later
in the proofs of Theorems \ref{theorem 1} and \ref{theorem 2}.
Both theorems offer a choice of signs, so for the sake of brevity
we present all our preparatory material in a form adapted to the case of upper signs.

\subsection{The unperturbed Dirac operator}
\label{The unperturbed Dirac operator}

Suppose that $\epsilon=0$, i.e. suppose that we are working with the standard
(un\-perturbed) metric. It is convenient to write the (massless) Dirac operator using the
triple of vector fields $(K_+)_j$, $j=1,2,3$,
defined in Appendix~\ref{Special vector fields on the 3-sphere}
as our frame,
 see subsection \ref{Definition via frames}
for the definition of a frame.
Straightforward calculations show that in this case the Dirac operator
reads
\begin{equation}
\label{unperturbed massless Dirac operator}
W^{(0)}=-is^j(L_+)_j+\frac32I\,,
\end{equation}
where $s^j$ are the standard Pauli matrices \eqref{Pauli matrices 2},
$(L_+)_j$ are the scalar first order linear differential operators
\eqref{definition of operator Lj}
and $I$ is the $2\times2$ identity matrix.
The superscript in $W^{(0)}$ indicates that the metric is unperturbed.

Let $v^{(0)}$ be a normalised eigenfunction corresponding to the eigenvalue $+\frac32$
of the unperturbed Dirac operator \eqref{unperturbed massless Dirac operator}.
For example, one can take
\begin{equation}
\label{unperturbed eigenfunction}
v^{(0)}=\frac1{\sqrt2\,\pi}
\begin{pmatrix}
1\\0
\end{pmatrix}.
\end{equation}
Here one can replace
$\begin{pmatrix}
1\\0
\end{pmatrix}$ by any other constant complex 2-column of unit norm.
The freedom in the choice of $v^{(0)}$ is due to the fact that $+\frac32$
is a double eigenvalue of the unperturbed Dirac operator.
The choice of a particular $v^{(0)}$ does not affect subsequent calculations,
what matters is that $v^{(0)}$ is a constant spinor.

Choosing the optimal frame (gauge)
is crucial for our subsequent arguments because we will heavily use the fact that the
eigenspinor $v^{(0)}$ of the unperturbed Dirac operator is a constant spinor.
See also Remark \ref{remarks for Killing vector fields}(c).

Observe that the triple of vector fields $(K_+)_j{}^\alpha$ uniquely defines a
triple of co\-vector fields $(K_+)^j{}_\alpha$:
the relation between the two is specified by the condition
\begin{equation}
\label{relation between standard frame and coframe}
(K_+)_j{}^\alpha(K_+)^k{}_\alpha=\delta_j{}^k,
\qquad j,k=1,2,3.
\end{equation}
Of course, the covector $(K_+)^j{}_\alpha$ is obtained by lowering
the tensor index in the vector $(K_+)_j{}^\alpha$ by means of the standard
metric on $\mathbb{S}^3$. Here the position of the frame index $j$, as a subscript
or superscript, is irrelevant.

In the next subsection we will make use of both
the vector fields $(K_+)_j{}^\alpha$
and
the covector fields $(K_+)^j{}_\alpha$.

\subsection{The perturbed Dirac operator}
\label{The perturbed Dirac operator}

Let $e_j{}^\alpha(y;\epsilon)$ be a frame corresponding to the perturbed metric
$g_{\alpha\beta}(y;\epsilon)$.
This frame can be written as
\begin{equation}
\label{expansion of frame over Killing frames}
e_j{}^\alpha(y;\epsilon)=(c_+)_j{}^k(y;\epsilon)\,(K_+)_k{}^\alpha(y)\,,
\end{equation}
where $(c_+)_j{}^k$ are some real scalar fields.
Without loss of generality we choose to work with frames satisfying
the symmetry condition
\begin{equation}
\label{symmetry condition}
(c_+)_j{}^k=(c_+)_k{}^j,
\end{equation}
which can always be achieved by an application of a gauge transformation ---
multiplication by a $3\times3$ orthogonal matrix-function.
Then formulae
\eqref{definition of tensor h},
\eqref{definition of tensor h frame},
\eqref{expansion of frame over Killing frames}
and
\eqref{symmetry condition}
imply
\begin{equation}
\label{frame in terms of h}
(c_+)_j{}^k(y;\epsilon)=\delta_j{}^k-\frac\epsilon2(h_+)_{jk}(y)+O(\epsilon^2)\,.
\end{equation}

The frame \eqref{expansion of frame over Killing frames}
uniquely defines the corresponding coframe $e^j{}_\alpha$
analogously to \eqref{relation between standard frame and coframe}:
\begin{equation}
\label{relation between general frame and coframe}
e_j{}^\alpha e^k{}_\alpha=\delta_j{}^k,
\qquad j,k=1,2,3.
\end{equation}
Of course, the covector $e^j{}_\alpha$ is obtained by lowering
the tensor index in the vector $e_j{}^\alpha$ by means of the perturbed
metric $g_{\alpha\beta}(y;\epsilon)$ on $\mathbb{S}^3$.
Formulae
\eqref{expansion of frame over Killing frames},
\eqref{relation between standard frame and coframe}
and
\eqref{relation between general frame and coframe}
imply
\begin{equation}
\label{expansion of coframe over Killing coframes}
e^j{}_\alpha(y;\epsilon)=(d_+)^j{}_k(y;\epsilon)\,(K_+)^k{}_\alpha(y)\,,
\end{equation}
where
\begin{equation}
\label{relation between coefficients c and d}
(c_+)_j{}^k(d_+)^l{}_k=\delta_j{}^l\,,
\qquad j,l=1,2,3.
\end{equation}
By
\eqref{relation between coefficients c and d}
and
\eqref{symmetry condition},
the matrix of scalar coefficients 
$(d_+)^j{}_k$ is also symmetric,
\begin{equation}
\label{symmetry condition for coefficients d}
(d_+)^j{}_k=(d_+)^k{}_j\,.
\end{equation}
Formulae
\eqref{frame in terms of h},
\eqref{relation between coefficients c and d},
\eqref{symmetry condition}
and
\eqref{symmetry condition for coefficients d}
give
\begin{equation}
\label{coframe in terms of h}
(d_+)^j{}_k(y;\epsilon)=\delta^j{}_k+\frac\epsilon2(h_+)_{jk}(y)+O(\epsilon^2)\,.
\end{equation}

Let $W(\epsilon)$ be the perturbed Dirac operator and let
$W_{1/2}(\epsilon)$ be the corresponding perturbed Dirac operator
on half-densities. According to
\eqref{definition of Weyl operator on half-densities},
the two are related as
\begin{equation}
\label{definition of Weyl operator on half-densities with epsilon}
W(\epsilon)=
(\rho(\epsilon))^{-1/2}
\,W_{1/2}(\epsilon)\,
(\rho(\epsilon))^{1/2}.
\end{equation}

\begin{lemma}
\label{lemma perturbed massless Dirac operator half}
The perturbed Dirac operator on half-densities $W_{1/2}(\epsilon)$
acts on 2-columns of complex-valued half-densities $v_{1/2}$ as
\begin{equation}
\label{lemma perturbed massless Dirac operator half eq 1}
\begin{split}
v_{1/2}\mapsto 
&-\frac i2s^j\,
(\rho_0)^{1/2}
\left[
(c_+)_j{}^k(L_+)_k
+
(L_+)_k(c_+)_j{}^k
\right]
(\rho_0)^{-1/2}\,v_{1/2}
\\
&\qquad+(W_{1/2}(\epsilon))_\mathrm{sub}\,v_{1/2}\,,
\end{split}
\end{equation}
where $(W_{1/2}(\epsilon))_\mathrm{sub}$ is its subprincipal symbol.
\end{lemma}

Let us emphasise that the Riemannian density appearing in
\eqref{lemma perturbed massless Dirac operator half eq 1}
is the un\-perturbed density $\rho_0$ and not the perturbed
density $\rho(\epsilon)$ as in
\eqref{definition of Weyl operator on half-densities with epsilon}.

\begin{proof}[Proof of Lemma \ref{lemma perturbed massless Dirac operator half}]
Formulae \eqref{definition of Weyl operator} and
\eqref{definition of Weyl operator on half-densities with epsilon}
tell us that the principal symbol of the operator
$W_{1/2}(\epsilon)$ is $\sigma^\alpha(y;\epsilon)p_\alpha\,$.
Using formulae
\eqref{Pauli matrices 1},
\eqref{expansion of frame over Killing frames}
and
\eqref{definition of operator Lj}
we can rewrite this principal symbol as
\begin{equation}
\label{lemma perturbed massless Dirac operator half eq 2}
-is^j\,(c_+)_j{}^k[(L_+)_k]_\mathrm{prin}\,.
\end{equation}
But
\eqref{lemma perturbed massless Dirac operator half eq 2}
is also the principal symbol of the operator
\eqref{lemma perturbed massless Dirac operator half eq 1},
so the proof reduces to proving that the operator
\begin{equation*}
v_{1/2}\mapsto 
-\frac i2s^j\,
(\rho_0)^{1/2}
\left[
(c_+)_j{}^k(L_+)_k
+
(L_+)_k(c_+)_j{}^k
\right]
(\rho_0)^{-1/2}\,v_{1/2}
\end{equation*}
has zero subprincipal symbol.
By \cite[Proposition 2.1.13]{mybook}
it is sufficient to prove that the operators
$(\rho_0)^{1/2}(L_+)_k(\rho_0)^{-1/2}$, $k=1,2,3$,
have zero subprincipal symbols.
But the latter is a consequence of
\eqref{definition of operator Lj}
and the fact that our $(K_+)_k{}^\alpha$,
being Killing vector fields with respect to the unperturbed metric,
are divergence-free.
\end{proof}

According to \cite[formulae (6.1) and (8.1)]{jst_part_b}
the explicit formula for the sub\-principal symbol of the
Dirac operator on half-densities reads
\begin{equation}
\label{subprincipal1}
(W_{1/2}(\epsilon))_\mathrm{sub}=If(\epsilon),
\end{equation}
where $I$ is the $2\times2$ identity matrix
and $f(\epsilon)=f(y;\epsilon)$
is the scalar function
\begin{equation}
\label{subprincipal2}
\begin{split}
f(\epsilon)
&:=\frac{\delta_{kl}}{4\rho(\epsilon)}
\left[
e^k{}_{1}
\,\frac{\partial e^l{}_{3}}{\partial y^2}
+
e^k{}_{2}
\,\frac{\partial e^l{}_{1}}{\partial y^3}
+
e^k{}_{3}
\,\frac{\partial e^l{}_{2}}{\partial y^1}
\right.
\\
&\qquad\qquad-\left.
e^k{}_{1}
\,\frac{\partial e^l{}_{2}}{\partial y^3}
-
e^k{}_{2}
\,\frac{\partial e^l{}_{3}}{\partial y^1}
-
e^k{}_{3}
\,\frac{\partial e^l{}_{1}}{\partial y^2}
\right],
\end{split}
\end{equation}
with $e^k{}_j=e^k{}_j(y;\epsilon)$.

Combining formulae
\eqref{definition of Weyl operator on half-densities with epsilon},
\eqref{lemma perturbed massless Dirac operator half eq 1},
\eqref{subprincipal1}
and
\eqref{subprincipal2}
we conclude that
the perturbed Dirac operator $W(\epsilon)$
acts on 2-columns of complex-valued scalar fields $v$ as
\begin{equation}
\label{perturbed massless Dirac operator}
v\mapsto 
-\frac i2s^j\,
\sqrt{\frac{\rho_0}{\rho(\epsilon)}}
\left[
(c_+)_j{}^k(L_+)_k
+
(L_+)_k(c_+)_j{}^k
\right]
\sqrt{\frac{\rho(\epsilon)}{\rho_0}}
\,v\,+\,f(\epsilon)\,v\,.
\end{equation}

Of course, when $\epsilon=0$
formulae
\eqref{perturbed massless Dirac operator}
and
\eqref{subprincipal2}
turn into formula
\eqref{unperturbed massless Dirac operator}
with $W^{(0)}=W(0)$.

\subsection{Half-densities versus scalar fields}
\label{Half-densities versus scalar fields}

Given a pair of 2-columns of complex-valued half-densities, $v_{1/2}$ and $w_{1/2}$,
we define their inner product as
\begin{equation}
\label{inner product on half-densities}
\langle v_{1/2}\,,w_{1/2}\rangle:=
\int_{\mathbb{S}^3}(w_{1/2})^*\,v_{1/2}\,dy\,.
\end{equation}
The advantage of \eqref{inner product on half-densities}
over
\eqref{inner product on 2-columns of scalar fields}
is that the inner product \eqref{inner product on half-densities}
does not depend on the metric. Consequently, if we work with half-densities,
perturbations of the metric will not change our Hilbert space.
And, unsurprisingly, the perturbation process described in \cite[Section 4]{torus}
was written in terms of half-densities.

The explicit formula for
the action of the operator $W_{1/2}(\epsilon)$
reads
\begin{equation}
\label{perturbed massless Dirac operator on half-densities}
v_{1/2}\,\mapsto\, 
-\frac i2s^j\,
(\rho_0)^{1/2}
\left[
(c_+)_j{}^k(L_+)_k
+
(L_+)_k(c_+)_j{}^k
\right]
(\rho_0)^{-1/2}\,v_{1/2}
\,+\,f(\epsilon)\,v_{1/2}\,,
\end{equation}
where $f(\epsilon)$ is the scalar function
\eqref{subprincipal2}.

Formulae
\eqref{perturbed massless Dirac operator on half-densities}
and
\eqref{subprincipal2}
give us a convenient explicit representation
of the perturbed Dirac operator on half-densities $W_{1/2}(\epsilon)$.
We will use this representation in the next two sections
when proving Theorems \ref{theorem 1} and \ref{theorem 2}.

When $\epsilon=0$
formulae
\eqref{perturbed massless Dirac operator on half-densities}
and
\eqref{subprincipal2}
turn into
\begin{equation*}
v_{1/2}\mapsto
-is^j(\rho_0)^{1/2}(L_+)_k(\rho_0)^{-1/2}v_{1/2}+\frac32v_{1/2}\,,
\end{equation*}
which is the action of
the unperturbed Dirac operator on half-densities $W^{(0)}_{1/2}=W_{1/2}(0)$.
The normalised eigenfunction of the operator $W^{(0)}_{1/2}$
corresponding to the eigenvalue $+\frac32$ reads
\begin{equation}
\label{unperturbed eigenfunction half-densities}
v_{1/2}^{(0)}=\rho_0^{1/2}v^{(0)},
\end{equation}
where $v^{(0)}$ is given by formula \eqref{unperturbed eigenfunction}.

\subsection{Asymptotic process}
\label{Asymptotic process}
Let us expand our Dirac operator on half-densities in powers of $\epsilon$,
\begin{equation}
\label{expansion for Dirac half-densities}
W_{1/2}(\epsilon)=W_{1/2}^{(0)}+\epsilon W_{1/2}^{(1)}
+\epsilon^2W_{1/2}^{(2)}+\ldots.
\end{equation}
Then, according to \cite[formulae (4.12) and (4.13)]{torus},
formula
\eqref{asymptotic expansion for smallest eigenvalues} holds with
\begin{align}
\label{abstract formula for lambda1}
\lambda_+^{(1)}&=
\bigl\langle\,
W_{1/2}^{(1)}\,v_{1/2}^{(0)}\,,v_{1/2}^{(0)}
\,\bigr\rangle,
\\
\label{abstract formula for lambda2}
\lambda_+^{(2)}&=
\bigl\langle\,
W_{1/2}^{(2)}\,v_{1/2}^{(0)}\,,v_{1/2}^{(0)}
\,\bigr\rangle
-
\bigl\langle\,
\bigl(W_{1/2}^{(1)}-\lambda_+^{(1)}I\bigr)
Q_{1/2}
\bigl(W_{1/2}^{(1)}-\lambda_+^{(1)}I\bigr)
\,v_{1/2}^{(0)}\,,v_{1/2}^{(0)}
\,\bigr\rangle,
\end{align}
where $Q_{1/2}$ is the pseudoinverse of the operator
$W^{(0)}_{1/2}-\frac32I$.
See \cite[Section 3]{torus} for definition of pseudoinverse.

\begin{lemma}
We have
\begin{equation}
\label{lemma about f eq 1}
\bigl\langle\,
W_{1/2}(\epsilon)\,v_{1/2}^{(0)}\,,v_{1/2}^{(0)}
\,\bigr\rangle
=
\bigl\langle\,
f(\epsilon)\,v_{1/2}^{(0)}\,,v_{1/2}^{(0)}
\,\bigr\rangle
=
\frac1{2\pi^2}\int_{\mathbb{S}^3}f(\epsilon)\,\rho_0\,dy\,.
\end{equation}
\end{lemma}

\begin{proof}
Substituting
\eqref{perturbed massless Dirac operator on half-densities},
\eqref{unperturbed eigenfunction half-densities}
and
\eqref{unperturbed eigenfunction}
into
the LHS of \eqref{lemma about f eq 1}
and using Remark
\ref{remarks for theorem 2}(d),
we see that the terms with $(L_+)_k$ integrate to zero, which leaves us with
the RHS of \eqref{lemma about f eq 1}.
\end{proof}

Let us now expand the scalar function $f(\epsilon)$ in powers of our $\epsilon$,
\begin{equation}
\label{expansion for scalar function f}
f(\epsilon)=f^{(0)}+\epsilon f^{(1)}
+\epsilon^2f^{(2)}+\ldots.
\end{equation}
Here, of course, $f^{(0)}=f(0)=\frac32$.

Formulae
\eqref{lemma about f eq 1},
\eqref{expansion for Dirac half-densities}
and 
\eqref{expansion for scalar function f}
imply
\begin{equation*}
\bigl\langle\,
W_{1/2}^{(n)}\,v_{1/2}^{(0)}\,,v_{1/2}^{(0)}
\,\bigr\rangle
=
\bigl\langle\,
f^{(n)}\,v_{1/2}^{(0)}\,,v_{1/2}^{(0)}
\,\bigr\rangle
=
\frac1{2\pi^2}\int_{\mathbb{S}^3}f^{(n)}\,\rho_0\,dy\,,
\qquad n=0,1,\ldots.
\end{equation*}
Then formulae
\eqref{abstract formula for lambda1}
and
\eqref{abstract formula for lambda2}
become
\begin{align}
\label{abstract formula for lambda1 simplified}
\lambda_+^{(1)}&=
\frac1{2\pi^2}\int_{\mathbb{S}^3}f^{(1)}\,\rho_0\,dy\,,
\\
\label{abstract formula for lambda2 simplified}
\begin{split}
\lambda_+^{(2)}&=
\frac1{2\pi^2}\int_{\mathbb{S}^3}f^{(2)}\,\rho_0\,dy\,
-
\bigl\langle\,
\bigl(W_{1/2}^{(1)}-\lambda_+^{(1)}I\bigr)
Q_{1/2}
\bigl(W_{1/2}^{(1)}-\lambda_+^{(1)}I\bigr)
\,v_{1/2}^{(0)}\,,v_{1/2}^{(0)}
\,\bigr\rangle\\
&=\frac1{2\pi^2}\int_{\mathbb{S}^3}f^{(2)}\,\rho_0\,dy\,
-\bigl\langle\,
Q_{1/2}
\bigl(W_{1/2}^{(1)}-\lambda_+^{(1)}I\bigr)
\,v_{1/2}^{(0)}\,,\bigl(W_{1/2}^{(1)}-\lambda_+^{(1)}I\bigr)v_{1/2}^{(0)}
\,\bigr\rangle.
\end{split}
\end{align}

\section{Proof of Theorem \ref{theorem 1}}
\label{Proof of Theorem 1}

We prove Theorem \ref{theorem 1} for the case of upper signs.

We have
\begin{equation}
\label{expansion for rho}
\rho(\epsilon)=
\rho_0\left(
1+\frac\epsilon2h_{\alpha\beta}\,(g_0)^{\alpha\beta}+O(\epsilon^2)
\right).
\end{equation}
Using formulae
\eqref{expansion of coframe over Killing coframes},
\eqref{symmetry condition for coefficients d}
and
\eqref{coframe in terms of h},
we get
\begin{equation}
\label{expansion for expression with square brackets}
\begin{split}
&\delta_{kl}
\left[
e^k{}_{1}
\,\frac{\partial e^l{}_{3}}{\partial y^2}
+
e^k{}_{2}
\,\frac{\partial e^l{}_{1}}{\partial y^3}
+
e^k{}_{3}
\,\frac{\partial e^l{}_{2}}{\partial y^1}
-
e^k{}_{1}
\,\frac{\partial e^l{}_{2}}{\partial y^3}
-
e^k{}_{2}
\,\frac{\partial e^l{}_{3}}{\partial y^1}
-
e^k{}_{3}
\,\frac{\partial e^l{}_{1}}{\partial y^2}
\right]
\\
&\qquad
=6\rho_0
\left(
1+\frac\epsilon3(h_+)_{jj}+O(\epsilon^2)
\right)
=6\rho_0
\left(
1+\frac\epsilon3 h_{\alpha\beta}\,(g_0)^{\alpha\beta}+O(\epsilon^2)
\right).
\end{split}
\end{equation}
Substitution of
\eqref{expansion for rho}
and
\eqref{expansion for expression with square brackets}
into
\eqref{subprincipal2}
gives us
\begin{equation}
\label{formula for f1}
f^{(1)}=-\frac14
h_{\alpha\beta}\,(g_0)^{\alpha\beta}.
\end{equation}

Finally, substituting
\eqref{formula for f1}
into
\eqref{abstract formula for lambda1 simplified}
and using
\eqref{formula for V1},
we arrive at
\eqref{formula for lambda1}.
\qed

\section{Proof of Theorem \ref{theorem 2}}
\label{Proof of Theorem 2}

We prove Theorem \ref{theorem 2} for the case of upper signs.

Recall also that we are proving this theorem under the
assumption \eqref{shear}. This implies, in particular, that
\begin{equation}
\label{lambda 1 is zero}
\lambda_+^{(1)}=0.
\end{equation}
With account of \eqref{lambda 1 is zero}, in order to use formula \eqref{abstract formula for lambda2 simplified}
we require the expressions for the scalar function $f^{(2)}$
and for
$W_{1/2}^{(1)}\,v_{1/2}^{(0)}\,$.

Substituting
\eqref{expansion of coframe over Killing coframes}
and
\eqref{coframe in terms of h}
into
\eqref{subprincipal2}
and using
\eqref{shear}
and
\eqref{symmetry condition for coefficients d},
we get
\begin{align}
\label{subprincipal98}
f^{(1)}&=0,
\\
\label{subprincipal99}
f^{(2)}&=
\frac14(h_+)_{jk}(h_+)_{jk}
-
\frac1{16}\varepsilon_{qks}(h_+)_{jq}\left[(L_+)_s(h_+)_{jk}\right].
\end{align}

Examination of formulae
\eqref{perturbed massless Dirac operator on half-densities},
\eqref{frame in terms of h}
and
\eqref{subprincipal98}
gives us
the explicit formula for
the action of the operator $W_{1/2}^{(1)}$\,:
\begin{equation}
\label{perturbed massless Dirac operator on half-densities 1}
v_{1/2}\mapsto 
\frac i4s^j\,
(\rho_0)^{1/2}
\left[
(h_+)_{jk}(L_+)_k
+
(L_+)_k(h_+)_{jk}
\right]
(\rho_0)^{-1/2}\,v_{1/2}.
\end{equation}
Acting with the operator
\eqref{perturbed massless Dirac operator on half-densities 1}
on the eigenfunction
\eqref{unperturbed eigenfunction half-densities}
of the unperturbed massless Dirac operator on half-densities,
we obtain
\begin{equation}
\label{perturbed massless Dirac operator on half-densities 2}
W_{1/2}^{(1)}\,v_{1/2}^{(0)}=
\frac i4
(\rho_0)^{1/2}\,
s^j\,
v^{(0)}
\left[
(L_+)_k(h_+)_{jk}
\right].
\end{equation}
Using formula
\eqref{perturbed massless Dirac operator on half-densities 2},
we get
\begin{equation}
\label{explicit formula for lambda2 hat auxiliary 4}
\begin{split}
&-
\bigl\langle\,
Q_{1/2}\,
W_{1/2}^{(1)}
\,v_{1/2}^{(0)}\,,
W_{1/2}^{(1)}\,v_{1/2}^{(0)}
\,\bigr\rangle
\\
&=
-\frac1{16}\int_{\mathbb{S}^3}
\bigl[(L_+)_r(h_+)_{qr}\bigr]
\left(
\left[\bigl[v^{(0)}\bigr]^*s^q\,(\rho_0)^{-1/2}\,Q_{1/2}\,(\rho_0)^{1/2}\,s^j\,v^{(0)}\right]
\bigl[(L_+)_k(h_+)_{jk}\bigr]
\right)
\rho_0\,dy\,.
\end{split}
\end{equation}
But $\,(\rho_0)^{-1/2}\,Q_{1/2}\,(\rho_0)^{1/2}=Q\,$,
the pseudoinverse of the operator
$W^{(0)}-\frac32I$.
Hence, formula \eqref{explicit formula for lambda2 hat auxiliary 4}
simplifies and reads now
\begin{equation}
\label{explicit formula for lambda2 hat auxiliary 4a}
\begin{split}
&-
\bigl\langle\,
Q_{1/2}\,
W_{1/2}^{(1)}
\,v_{1/2}^{(0)}\,,
W_{1/2}^{(1)}\,v_{1/2}^{(0)}
\,\bigr\rangle
\\
&=
-\frac1{16}\int_{\mathbb{S}^3}
\bigl[(L_+)_r(h_+)_{qr}\bigr]
\left(
\left[\bigl[v^{(0)}\bigr]^*s^q\,Q\,s^j\,v^{(0)}\right]
\bigl[(L_+)_k(h_+)_{jk}\bigr]
\right)
\rho_0\,dy\,.
\end{split}
\end{equation}

Observe now that we have the identity
\begin{equation}
\label{identity for the square of Dirac}
\left(
W^{(0)}-\frac12I
\right)^2
=(-\Delta+1)I,
\end{equation}
where $I$ is the $2\times2$ identity matrix
and $\Delta$ is the Laplacian on scalar fields over $\mathbb{S}^3$
with standard metric.
Formula
\eqref{identity for the square of Dirac}
can be established by direct substitution
of \eqref{unperturbed massless Dirac operator}
and the use of the commutator formula \eqref{commutator identities}.
Formula
\eqref{identity for the square of Dirac}
appears also as Lemma 2 in \cite{bar_1996}.

Formula \eqref{identity for the square of Dirac} implies
\begin{equation}
\label{relation between two pseudoinverses}
Q=(-\Delta)^{-1}
\left(
W^{(0)}+\frac12I
\right)
=
(-\Delta)^{-1}
\left(
-is^l(L_+)_l+2I
\right).
\end{equation}
Formula \eqref{relation between two pseudoinverses},
in turn, gives us the following representation for the scalar pseudodifferential
operator $\bigl[v^{(0)}\bigr]^*s^q\,Q\,s^j\,v^{(0)}$:
\begin{equation}
\label{complicated scalar operator}
\begin{split}
&\bigl[v^{(0)}\bigr]^*s^q\,Q\,s^j\,v^{(0)}
\\
&=
2
\left(
\bigl[v^{(0)}\bigr]^*s^q\,s^j\,v^{(0)}
\right)
(-\Delta)^{-1}
-i
\left(
\bigl[v^{(0)}\bigr]^*s^q\,s^l\,s^j\,v^{(0)}
\right)
(-\Delta)^{-1}(L_+)_l\,.
\end{split}
\end{equation}

Substituting
\eqref{complicated scalar operator}
into
\eqref{explicit formula for lambda2 hat auxiliary 4a},
we get
\begin{equation}
\label{explicit formula for lambda2 hat auxiliary 5}
\begin{split}
-
\bigl\langle\,
&Q_{1/2}\,
W_{1/2}^{(1)}
\,v_{1/2}^{(0)}\,,
W_{1/2}^{(1)}\,v_{1/2}^{(0)}
\,\bigr\rangle\\
&=
-\frac18
\left(
\bigl[v^{(0)}\bigr]^*s^q\,s^j\,v^{(0)}
\right)
\int_{\mathbb{S}^3}
\bigl[(L_+)_r(h_+)_{qr}\bigr]
\bigl[
(-\Delta)^{-1}(L_+)_k(h_+)_{jk}\bigr]
\rho_0\,dy
\\
&+\frac1{16}
\left(
i
\bigl[v^{(0)}\bigr]^*s^q\,s^l\,s^j\,v^{(0)}
\right)
\int_{\mathbb{S}^3}
\bigl[(L_+)_r(h_+)_{qr}\bigr]
\bigl[(-\Delta)^{-1}(L_+)_l(L_+)_k(h_+)_{jk}\bigr]
\rho_0\,dy
\\
&=
\frac18
\operatorname{Re}
\left(
\bigl[v^{(0)}\bigr]^*s^q\,s^j\,v^{(0)}
\right)
\int_{\mathbb{S}^3}
(h_+)_{qr}
\bigl[
(-\Delta)^{-1}(L_+)_r(L_+)_k(h_+)_{jk}\bigr]
\rho_0\,dy
\\
&-\frac1{16}
\operatorname{Re}
\left(
i
\bigl[v^{(0)}\bigr]^*s^q\,s^l\,s^j\,v^{(0)}
\right)
\int_{\mathbb{S}^3}
(h_+)_{qr}
\bigl[(-\Delta)^{-1}(L_+)_r(L_+)_l(L_+)_k(h_+)_{jk}\bigr]
\rho_0\,dy\,.
\end{split}
\end{equation}

But
\[
\begin{split}
\operatorname{Re}
\left(
\bigl[v^{(0)}\bigr]^*s^q\,s^j\,v^{(0)}
\right)
&=
\frac12
\left(
\bigl[v^{(0)}\bigr]^*
\bigl(s^qs^j+s^js^q\bigr)
v^{(0)}
\right)\\
&=
\delta^{qj}
\left(
\bigl[v^{(0)}\bigr]^*
\,I\,
v^{(0)}
\right)
=
\frac1{2\pi^2}\,
\delta^{qj},
\\
\operatorname{Re}
\left(
i
\bigl[v^{(0)}\bigr]^*s^q\,s^l\,s^j\,v^{(0)}
\right)
&=
\frac i2
\left(
\bigl[v^{(0)}\bigr]^*
\bigl(s^qs^ls^j-s^js^ls^q\bigr)
v^{(0)}
\right)\\
&=-\varepsilon^{qlj}
\left(
\bigl[v^{(0)}\bigr]^*
\,I\,
v^{(0)}
\right)
=-\frac1{2\pi^2}\,\varepsilon^{qlj},
\end{split}
\]
where we made use of \eqref{unperturbed eigenfunction}.
Hence, formula \eqref{explicit formula for lambda2 hat auxiliary 5}
simplifies and reads now
\begin{equation}
\label{explicit formula for lambda2 hat auxiliary 6}
\begin{split}
-
\bigl\langle\,
&Q_{1/2}\,
W_{1/2}^{(1)}
\,v_{1/2}^{(0)}\,,
W_{1/2}^{(1)}\,v_{1/2}^{(0)}
\,\bigr\rangle
=
\frac1{16\pi^2}
\int_{\mathbb{S}^3}
(h_+)_{jr}
\bigl[
(-\Delta)^{-1}(L_+)_r(L_+)_k(h_+)_{jk}\bigr]
\rho_0\,dy
\\
&\qquad+\frac1{32\pi^2}\,
\varepsilon_{qlj}
\int_{\mathbb{S}^3}
(h_+)_{qr}
\bigl[(-\Delta)^{-1}(L_+)_r(L_+)_l(L_+)_k(h_+)_{jk}\bigr]
\rho_0\,dy\,.
\end{split}
\end{equation}

Substituting
\eqref{subprincipal99}
and
\eqref{explicit formula for lambda2 hat auxiliary 6}
into
\eqref{abstract formula for lambda2 simplified}
we arrive at
\eqref{formula for lambda2},
\eqref{formula for lambda2 quadratic form}
with upper signs.
\qed

\section{Generalized Berger spheres}
\label{Generalized Berger spheres}

A left-handed generalized Berger sphere is a 3-sphere equipped with metric
\begin{equation}
\label{generalized Berger metric 1}
g_{\alpha\beta}=
C_{jk}(K_+)^j{}_\alpha (K_+)^k{}_\beta\,,
\end{equation}
where $C=(C_{jk})_{j,k=1}^3$ is a
\textbf{constant} $3\times3$ positive real symmetric matrix
and  $(K_+)^j{}_\alpha$, $j=1,2,3$, are our special covector fields defined in
accordance with Section~\ref{Main results} and formula
\eqref{relation between standard frame and coframe}.
One can, of course, define in a similar fashion right-handed generalized Berger spheres:
these involve the covector fields $(K_-)^j{}_\alpha$, $j=1,2,3$.
However, in this paper, as in \cite{hitchin}, we restrict our analysis to left-handed ones.

One can always perform a rotation in $\mathbb{R}^4$ so that
\eqref{generalized Berger metric 1} turns to
\begin{equation}
\label{generalized Berger metric 2}
g_{\alpha\beta}=
\sum_{j=1}^3a_j^2(K_+)^j{}_\alpha (K_+)^j{}_\beta
\,,
\end{equation}
where $a_j$, $j=1,2,3$, are some positive constants.
In formula \eqref{generalized Berger metric 2} the $(K_+)^j{}_\alpha$, $j=1,2,3$,
are new covector fields defined in the new Cartesian coordinate system
in accordance with formulae
\eqref{Killing vector fields in R4}
and
\eqref{relation between standard frame and coframe}.
Of course, $a_j^2$ are the eigenvalues of the matrix $C$.
Further on we assume that our generalized Berger metric has the form
\eqref{generalized Berger metric 2}.

To the authors' knowledge, metrics of the type 
\eqref{generalized Berger metric 2}
were first considered in Section 3 of \cite{hitchin}.
The expression ``generalized Berger sphere'' first appears in \cite{godbout}.
The standard (as opposed to the generalized) Berger sphere
corresponds to the case $a_2=a_3=1$,
and the standard sphere corresponds to the case $a_1=a_2=a_3=1$.

For future reference, let us give the formula for
the Riemannian volume \eqref{definition of volume}
of the generalized Berger sphere:
\begin{equation}
\label{volume generalized Berger sphere}
V=2\pi^2a_1a_2a_3\,.
\end{equation}

\subsection{Dirac operator on generalized Berger spheres}
\label{Dirac operator on generalized Berger spheres}
The remarkable feature of generalized Berger spheres is that for these metrics
the calculation of eigenvalues of the (massless) Dirac operator reduces to finding
roots of polynomials.

The Dirac operator
\eqref{perturbed massless Dirac operator}
corresponding to the generalized Berger metric reads
\begin{equation}
\label{Dirac for generalized Berger metric}
W=
-i
\sum_{j=1}^3
\frac1{a_j}s^j(L_+)_j
+
\nu I\,,
\end{equation}
where
\begin{equation}
\label{lowest positive eigenvalue for generalized Berger metric}
\nu=\frac{a_1^2+a_2^2+a_3^2}{2a_1a_2a_3}\,.
\end{equation}
In writing \eqref{Dirac for generalized Berger metric}
we followed the convention of choosing the symmetric gauge,
see formulae
\eqref{symmetry condition}
and
\eqref{symmetry condition for coefficients d}.
The constant \eqref{lowest positive eigenvalue for generalized Berger metric}
was written down by means of a careful application of formula
\eqref{subprincipal2}.

Note that formula \eqref{Dirac for generalized Berger metric}
appears also in Proposition 3.1 of \cite{hitchin}.

Examination of formula
\eqref{Dirac for generalized Berger metric}
shows that $\lambda=\nu$
is an eigenvalue of the Dirac operator,
with the corresponding eigenspinors being constant spinors.

In order to calculate other eigenvalues of the Dirac operator
it is convenient to extend our spinor field from $\mathbb{S}^3$ to a neighbourhood
of $\mathbb{S}^3$ in $\mathbb{R}^4$ and rewrite the operator in Cartesian
coordinates. Substituting
\eqref{Killing vector fields in R4}
into
\eqref{Dirac for generalized Berger metric},
we get
\begin{equation}
\label{Dirac for generalized Berger metric Cartesian 1}
\mathbf{W}=
-i
\sum_{j=1}^3
\frac1{a_j}s^j(\mathbf{L}_+)_j
+
\nu I\,,
\end{equation}
where
\begin{equation}
\label{Dirac for generalized Berger metric Cartesian 2}
\begin{aligned}
(\mathbf{L}_+)_1&=
-\mathbf{x}^4\partial_1-\mathbf{x}^3\partial_2+\mathbf{x}^2\partial_3+\mathbf{x}^1\partial_4
\,,
\\
(\mathbf{L}_+)_2&=
{\ \ }
\mathbf{x}^3\partial_1-\mathbf{x}^4\partial_2-\mathbf{x}^1\partial_3+\mathbf{x}^2\partial_4
\,,
\\
(\mathbf{L}_+)_3&=
-\mathbf{x}^2\partial_1+\mathbf{x}^1\partial_2-\mathbf{x}^4\partial_3+\mathbf{x}^3\partial_4
\,.
\end{aligned}
\end{equation}
Here the way to work with the Cartesian representation of the Dirac operator
is to act with
\eqref{Dirac for generalized Berger metric Cartesian 1},
\eqref{Dirac for generalized Berger metric Cartesian 2}
on a spinor field defined in a neighbourhood
of $\mathbb{S}^3$ and then restrict the result to \eqref{definition of 3-sphere}.
It is easy to see that under this procedure
the resulting spinor field on $\mathbb{S}^3$ does not depend on the way 
we extended our original spinor field from $\mathbb{S}^3$ to a neighbourhood
of $\mathbb{S}^3$ in $\mathbb{R}^4$.

The operators
\eqref{Dirac for generalized Berger metric Cartesian 2}
commute with the scalar Laplacian in $\mathbb{R}^4$.
This implies that these operators map homogeneous harmonic polynomials of degree $k$
to homogeneous harmonic polynomials of degree less than or equal to $k$.
Hence, the eigenspinors of the Dirac operator can be written in terms of
homogeneous harmonic polynomials.
Of course, the restriction of
homogeneous harmonic polynomials to the 3-sphere \eqref{definition of 3-sphere}
gives spherical functions, but we find working with polynomials in Cartesian coordinates
more convenient than working with spherical functions
in spherical coordinates \eqref{spherical coordinates for 3-sphere}.

Let us seek an eigenspinor which is linear in Cartesian coordinates
$\mathbf{x}^{\boldsymbol\alpha}$, $\boldsymbol\alpha=1,2,3,4$.
Such an eigenspinor is determined by eight complex constants and finding
the corresponding eigenvalues reduces to finding the eigenvalues of a particular
$8\times8$ Hermitian matrix. Explicit calculations (which we omit for the sake of brevity)
show that the characteristic polynomial
of this $8\times8$ Hermitian matrix is the square of a polynomial
of degree four whose roots are
\begin{equation}
\label{lowest negative eigenvalue for generalized Berger metric}
\nu-\frac1{a_1}-\frac1{a_2}-\frac1{a_3}\,,
\end{equation}
\begin{equation}
\label{lowest negative eigenvalue for generalized Berger metric extra 1}
\nu-\frac1{a_1}+\frac1{a_2}+\frac1{a_3}\,,
\end{equation}
\begin{equation}
\label{lowest negative eigenvalue for generalized Berger metric extra 2}
\nu+\frac1{a_1}-\frac1{a_2}+\frac1{a_3}\,,
\end{equation}
\begin{equation}
\label{lowest negative eigenvalue for generalized Berger metric extra 3}
\nu+\frac1{a_1}+\frac1{a_2}-\frac1{a_3}\,.
\end{equation}

One can repeat the above procedure for
homogeneous harmonic polynomials of degree $n=2,3,\ldots$,
thus reducing the problem of finding eigenvalues of the Dirac
operator on a generalized Berger sphere to finding roots of polynomials.
See Appendix
\ref{Explicit formulae for some eigenvalues of Dirac operator on generalized Berger spheres}
for further details.

Note that for the standard Berger sphere ($a_2=a_3=1$)
the spectrum of the Dirac operator was previously calculated in \cite{bar_1992,bar_1996a}.

\subsection{Testing Theorem~\ref{theorem 1} on generalized Berger spheres}
\label{Testing Theorem 1 on generalized Berger spheres}
From now on we will assume that the positive constants
$a_j$ are close to 1.
This assumption will allow us to identify the lowest, in terms of modulus,
positive and negative eigenvalues of the Dirac operator.

When $a_1=a_2=a_3=1$
the expression
\eqref{lowest positive eigenvalue for generalized Berger metric}
takes the value $+\frac32\,$
and
the expression
\eqref{lowest negative eigenvalue for generalized Berger metric}
takes the value $-\frac32\,$.
Hence,
\begin{equation}
\label{lowest positive eigenvalue for generalized Berger metric with lambda}
\lambda_+=\nu
\end{equation}
is the lowest positive eigenvalue of the Dirac operator
and
\begin{equation}
\label{lowest negative eigenvalue for generalized Berger metric with lambda}
\lambda_-=\nu-\frac1{a_1}-\frac1{a_2}-\frac1{a_3}
\end{equation}
is the lowest, in terms of modulus,
negative eigenvalue of the Dirac operator.
Recall that $\nu$ is given by formula
\eqref{lowest positive eigenvalue for generalized Berger metric}.
As for the expressions
\eqref{lowest negative eigenvalue for generalized Berger metric extra 1}--\eqref{lowest negative eigenvalue for generalized Berger metric extra 3},
their values are close to $+\frac52\,$.

In this subsection and the next one we assume that the constants $a_j$ appearing in
formula \eqref{generalized Berger metric 2} are smooth functions of the small parameter
$\epsilon$ and that $a_j(0)=1$.

Expanding
\eqref{volume generalized Berger sphere},
\eqref{lowest positive eigenvalue for generalized Berger metric with lambda}
and
\eqref{lowest negative eigenvalue for generalized Berger metric with lambda}
in powers of $\epsilon$, we get
\begin{equation}
\label{volume generalized Berger sphere expansion}
V(\epsilon)=2\pi^2\left(1+(a_1'+a_2'+a_3')\epsilon+O(\epsilon^2)\right),
\end{equation}
\begin{equation}
\label{lowest eigenvalue for generalized Berger metric expansion order 1}
\lambda_\pm(\epsilon)=\pm\frac32\mp\frac12(a_1'+a_2'+a_3')\epsilon+O(\epsilon^2)\,,
\end{equation}
where
\begin{equation*}
a_j':=\left.\frac{da_j}{d\epsilon}\right|_{\epsilon=0}\,.
\end{equation*}
Formulae
\eqref{asymptotic expansion for smallest eigenvalues},
\eqref{asymptotic expansion for volume},
\eqref{formula for V0},
\eqref{volume generalized Berger sphere expansion}
and
\eqref{lowest eigenvalue for generalized Berger metric expansion order 1}
imply \eqref{formula for lambda1}. Thus, we are in agreement with
Theorem~\ref{theorem 1}.

\subsection{Testing Theorem~\ref{theorem 2} on generalized Berger spheres}
\label{Testing Theorem 2 on generalized Berger spheres}

In this subsection we make the additional assumption
\begin{equation}
\label{abs additional assumption}
a_1(\epsilon)\,a_2(\epsilon)\,a_3(\epsilon)=1\,,
\end{equation}
which ensures the preservation of Riemannian volume
\eqref{volume generalized Berger sphere}
under perturbations.
But generalized Berger spheres are homogeneous Riemannian spaces,
so preservation of Riemannian volume is equivalent
to preservation of Riemannian density.
Hence, \eqref{abs additional assumption}
implies \eqref{shear},
which is required for testing Theorem~\ref{theorem 2}.

For future reference note that formula
\eqref{abs additional assumption}
implies
\begin{equation}
\label{abs additional assumption consequence order 1}
a_1'+a_2'+a_3'=0\,,
\end{equation}
\begin{equation}
\label{abs additional assumption consequence order 2}
a_1''+a_2''+a_3''+2(a_1'a_2'+a_2'a_3'+a_3'a_1')=0\,,
\end{equation}
where
\begin{equation}
\label{abs double prime}
a_j'':=\left.\frac{d^2a_j}{d\epsilon^2}\right|_{\epsilon=0}\,.
\end{equation}

Expanding
\eqref{lowest positive eigenvalue for generalized Berger metric with lambda}
and
\eqref{lowest negative eigenvalue for generalized Berger metric with lambda}
in powers of $\epsilon$
and using formulae
\eqref{abs additional assumption consequence order 1}
and
\eqref{abs additional assumption consequence order 2},
we get
\begin{equation}
\label{lowest positive eigenvalue for generalized Berger metric expansion order 2}
\lambda_+(\epsilon)=\frac32
+\left(
(a_1')^2+(a_2')^2+(a_3')^2
\right)
\epsilon^2
+O(\epsilon^3)\,,
\end{equation}
\begin{equation}
\label{lowest negative eigenvalue for generalized Berger metric expansion order 2}
\lambda_-(\epsilon)=-\frac32
+\frac12
\left(
(a_1')^2+(a_2')^2+(a_3')^2
\right)
\epsilon^2
+O(\epsilon^3)\,.
\end{equation}
Note that the second derivatives
\eqref{abs double prime}
do not appear in formulae
\eqref{lowest positive eigenvalue for generalized Berger metric expansion order 2}
and
\eqref{lowest negative eigenvalue for generalized Berger metric expansion order 2},
which is in agreement with
Remark~\ref{remarks for theorem 2}(j).

We first test whether formula
\eqref{lowest positive eigenvalue for generalized Berger metric expansion order 2}
agrees with Theorem~\ref{theorem 2}.
Calculating the scalars \eqref{definition of tensor h frame} with upper sign, we get
\begin{equation}
\label{h plus Berger}
(h_+)_{jk}=
2\sum_{l=1}^3
a_l'\,\delta_{lj}\,\delta_{lk}\,.
\end{equation}
Formulae
\eqref{formula for lambda2 quadratic form}
and
\eqref{h plus Berger}
imply
\begin{equation}
\label{Q plus Berger}
P_+=
(a_1')^2+(a_2')^2+(a_3')^2.
\end{equation}
Substituting
\eqref{Q plus Berger}
into
\eqref{formula for lambda2}
and using
\eqref{formula for V0},
we get
$\lambda_+^{(2)}=(a_1')^2+(a_2')^2+(a_3')^2$,
which is in agreement with
\eqref{lowest positive eigenvalue for generalized Berger metric expansion order 2}.

In the remainder of this subsection we test whether formula
\eqref{lowest negative eigenvalue for generalized Berger metric expansion order 2}
agrees with Theorem~\ref{theorem 2}. This is trickier because
the scalar fields $(h_-)_{jk}$
are not constant.

Consider the matrix-function 
\begin{equation*}
\begin{pmatrix}
(\mathbf{x}^1)^2-(\mathbf{x}^2)^2-(\mathbf{x}^3)^2+(\mathbf{x}^4)^2&
2(\mathbf{x}^1\mathbf{x}^2-\mathbf{x}^3\mathbf{x}^4)&
2(\mathbf{x}^1\mathbf{x}^3+\mathbf{x}^2\mathbf{x}^4)\\
2(\mathbf{x}^1\mathbf{x}^2+\mathbf{x}^3\mathbf{x}^4)&
-(\mathbf{x}^1)^2+(\mathbf{x}^2)^2-(\mathbf{x}^3)^2+(\mathbf{x}^4)^2&
2(\mathbf{x}^2\mathbf{x}^3-\mathbf{x}^1\mathbf{x}^4)\\
2(\mathbf{x}^1\mathbf{x}^3-\mathbf{x}^2\mathbf{x}^4)&
2(\mathbf{x}^1\mathbf{x}^4+\mathbf{x}^2\mathbf{x}^3)&
-(\mathbf{x}^1)^2-(\mathbf{x}^2)^2+(\mathbf{x}^3)^2+(\mathbf{x}^4)^2
\end{pmatrix}
\end{equation*}
whose elements are homogeneous harmonic quadratic polynomials.
Let $O$ be the restriction of the above matrix-function
to the 3-sphere \eqref{definition of 3-sphere}.
Note that the matrix-function $O$ is orthogonal.
Let us denote the elements the matrix-function $O$ by $O_{jk}$,
with the first subscript enumerating rows and the second enumerating columns.
The two sets of scalar fields, $(h_+)_{jk}$ and $(h_-)_{jk}\,$,
are related as
\begin{equation}
\label{h minus in terms of h plus without matrix notation}
(h_-)_{il}=
O_{ij}
(h_+)_{jk}
O_{lk}\,.
\end{equation}
Substitution of \eqref{h plus Berger}
into
\eqref{h minus in terms of h plus without matrix notation}
gives us explicit formulae for the scalar fields $(h_-)_{jk}\,$.

We now need to substitute
\eqref{h minus in terms of h plus without matrix notation}
into the formula for $P_-\,$, see
\eqref{formula for lambda2 quadratic form}.

Observe that the (spherical) functions $O_{jk}$ satisfy the identity
\begin{equation}
\label{identity for spherical functions}
(L_-)_i\,O_{jk}=2\varepsilon_{ijl}\,O_{lk}\,.
\end{equation}
Formulae
\eqref{h minus in terms of h plus without matrix notation}
and
\eqref{identity for spherical functions}
and the fact that the matrix of constants $(h_+)_{jk}$ is symmetric
imply $(L_-)_s(L_-)_j(h_-)_{jk}=0$, so the last two terms in the RHS of
\eqref{formula for lambda2 quadratic form}
vanish, giving us
\begin{equation}
\label{auxiliary 1}
P_-=-\frac14(h_-)_{jk}(h_-)_{jk}
-\frac1{16}\varepsilon_{qks}(h_-)_{jq}\left[(L_-)_s(h_-)_{jk}\right].
\end{equation}

We examine the two terms in the RHS of
\eqref{auxiliary 1} separately.
As the matrix $O$ is orthogonal, we have, with account of \eqref{h plus Berger},
\begin{equation}
\label{auxiliary 2}
-\frac14(h_-)_{jk}(h_-)_{jk}
=
-\frac14(h_+)_{jk}(h_+)_{jk}
=
-\bigl[(a_1')^2+(a_2')^2+(a_3')^2\bigr].
\end{equation}
The other term is evaluated by substituting
\eqref{h minus in terms of h plus without matrix notation},
using the identity
\eqref{identity for spherical functions}
and the fact that our perturbation of the metric is pointwise trace-free
$(h_\pm)_{jj}=0$,
which gives us
\begin{equation}
\label{auxiliary 3}
-\frac1{16}\varepsilon_{qks}(h_-)_{jq}\left[(L_-)_s(h_-)_{jk}\right]
=
\frac38(h_+)_{jk}(h_+)_{jk}
=
\frac32\bigl[(a_1')^2+(a_2')^2+(a_3')^2\bigr].
\end{equation}
Substituting
\eqref{auxiliary 2}
and
\eqref{auxiliary 3}
into the RHS of
\eqref{auxiliary 1},
we arrive at
\begin{equation}
\label{Q minus Berger}
P_-=
\frac12
\bigl[
(a_1')^2+(a_2')^2+(a_3')^2
\bigr].
\end{equation}
Substituting
\eqref{Q minus Berger}
into
\eqref{formula for lambda2}
and using
\eqref{formula for V0},
we get
$\lambda_-^{(2)}=\frac12[(a_1')^2+(a_2')^2+(a_3')^2]$,
which is in agreement with
\eqref{lowest negative eigenvalue for generalized Berger metric expansion order 2}.

\section*{Acknowledgments}

The authors are grateful to Jason Lotay for advice regarding Berger spheres.


\begin{appendices}

\section{Orientation}
\label{Orientation}

The unit 3-sphere, $\mathbb{S}^3$, is the hypersurface in $\mathbb{R}^4$ defined
by the equation
\begin{equation}
\label{definition of 3-sphere}
\|\mathbf{x}\|=1,
\end{equation}
where $\|\cdot\|$ is the standard Euclidean norm.
Spherical coordinates
\begin{equation}
\label{spherical coordinates for 3-sphere}
\begin{pmatrix}
\mathbf{x}^1\\
\mathbf{x}^2\\
\mathbf{x}^3\\
\mathbf{x}^4
\end{pmatrix}
=
\begin{pmatrix}
\cos y^1\\
\sin y^1\cos y^2\\
\sin y^1\sin y^2\cos y^3\\
\sin y^1\sin y^2\sin y^3\\
\end{pmatrix},
\quad
y^1,y^2\in(0,\pi),
\quad
y^3\in[0,2\pi),
\end{equation}
are an example of
local coordinates on $\mathbb{S}^3$.
We define the orientation of spherical
coordinates \eqref{spherical coordinates for 3-sphere}
to be positive.

\section{The Dirac operator}
\label{The Dirac operator}

\subsection{Classical geometric definition}
\label{Classical geometric definition}

Unlike the rest of the paper, in this subsection we work in a more general setting.
Namely, we do not assume our base manifold to be 3-dimensional.

The material presented in this subsection can be found in many classical books on spin geometry. We follow the notation from \cite{Lawson and Michelsohn}. Let $X$ be an $m$-dimensional connected manifold and $E$ be an $n$-dimensional oriented Riemannian vector bundle with a spin structure. Recall that a complex spin bundle of $E$ is given by $S_{\mathbb{C}}(E)=P_{\mathrm{Spin}}(E)\times_\mu M_{\mathbb{C}}$, where $P_{\mathrm{Spin}}(E)$ is the principal $\mathrm{Spin}_n$ bundle associated with $E$, $M_{\mathbb{C}}$ is an $N$-dimensional left complex module for $\mathbb{C}l(\mathbb{R}^n)=Cl(\mathbb{R}^n)\otimes \mathbb{C}$ and $\mu:\mathrm{Spin}_n \mapsto \mathrm{SO}(M_{\mathbb{C}})$ is the representation induced by left multiplication by elements of $\mathrm{Spin}_n\subset Cl^0(\mathbb{R}^n)\subset \mathbb{C}l(\mathbb{R}^n)$. 
\begin{theorem}[{\cite[Chapter II, Section 4]{Lawson and Michelsohn}}]
	\label{Lawson and Michelsohn thm}
Let $\omega$ be a connection $1$-form on the bundle of $P_{\mathrm{SO}}(E)$ oriented orthonormal bases of $E$, which can be expressed as $\omega=\sum_{i<j}\omega_{ij}\,e_i\wedge e_j$. Here $e_i\wedge e_j$ is the elementary skew-symmetric $(i,j)$ matrix. Then the covariant derivative $\nabla^s$ on $S(E)$ is given locally by the formula
\begin{equation}
\label{Lawson's covariant derivatives 1}
\nabla^s b_\alpha=\frac{1}{2} \sum_{i<j}\tilde{\omega}_{ij}\otimes e_i e_j \cdot b_\alpha\,,
\end{equation}
where $\mathscr{E}=(e_1,\dots,e_n)$ is a local section of $P_{\mathrm{SO}}(E)$ on $U$, $\tilde{\omega}=\mathscr{E}^*(\omega)$, and $(b_1,\dots,b_N)$ is a local section of $P_{\mathrm{SO}}(S(E))$.
\end{theorem}
\begin{remark}
\label{Lawson remark 1}
A local section
$(b_1,\dots,b_N)$ corresponds to the choice of basis in the module $M_\mathbb{C}\,$, that is, a basis in the representation $\mu$. Once the representation $\mu$ and basis $(b_1,\dots,b_N)$ are chosen, spinor fields $\varphi$ can be written locally as
\[
\varphi=\varphi^\alpha b_\alpha\,,
\]
where $\varphi^\alpha \in C^\infty(U)$. We also have
\[
\mu(e_j)\,b_\alpha=\mu(e_j)^{\beta}{}_{\alpha}\,b_\beta\,.
\]
Hence, formula \eqref{Lawson's covariant derivatives 1} is equivalent to
\begin{equation*}
\label{Lawson's covariant derivatives 2}
\nabla^s \varphi^\alpha=\mathrm{d}\varphi^\alpha+\frac{1}{2} \sum_{i<j}\tilde{\omega}_{ij}\otimes \mu(e_i e_j)^\alpha{}_\beta\,\varphi^\beta\,,
\end{equation*}
or abbreviated as
\begin{equation}
\label{Lawson's covariant derivatives 3}
\nabla^s \varphi=\mathrm{d}\varphi+\frac{1}{2} \sum_{i<j}\tilde{\omega}_{ij}\otimes \mu(e_i e_j) \cdot \varphi\,.
\end{equation}
\end{remark}
The (massless) Dirac operator acting on $S(E)$ is defined as
\begin{equation}
\label{geometric Dirac operator}
D^s\varphi=e_i\nabla^s_{e_i}\varphi\,.
\end{equation}

It is well known that any spinor bundle of $E$ on a connected manifold can be decomposed into a direct sum of irreducible spinor bundles. However, the most interesting case is when $E=TX$, which implies $n=m$.
Therefore, we will only consider the irreducible complex spinor bundle on $TX$, which further implies that the module has complex dimension $N=2^{\lfloor\frac{n}{2} \rfloor}$, where $\lfloor\,\cdot\,\rfloor$ stands for the integer part.

Let $\Delta_n^\mathbb{C}$ be the representation of $\mathrm{Spin}_n$ given by restricting an irreducible complex representation $\mathbb{C}l(\mathbb{R}^n)\mapsto\mathrm{Hom}_\mathbb{C}(S_N,S_N)$ to $\mathrm{Spin}_n\subset Cl^0(\mathbb{R}^n)\subset \mathbb{C}l(\mathbb{R}^n)$.
\begin{remark}
	\label{Lawson remark 2}
When $n$ is odd, this representation of $\mathrm{Spin}_n$ is independent of which irreducible representations of $\mathbb{C}l(\mathbb{R}^n)$ are used.
\end{remark}

As we only work in odd dimension(s),
we focus on the spin bundle
$\Delta_n(X):=P_{\mathrm{Spin}}(TX)\times_{\Delta_n^\mathbb{C}} S_N$.
Furthermore, we consider the Levi-Civita connection on $P_{\mathrm{SO}}(TM)$, which induces a connection on $P_{\mathrm{Spin}}(TM)$ as given in Theorem \ref{Lawson and Michelsohn thm}. We use $\nabla$ to denote covariant derivatives induced by the Levi-Civita connection. In this case the $\omega_{ij}$ in Theorem \ref{Lawson and Michelsohn thm} are given by
\begin{equation}
\label{Levi-Civita 1 form}
\omega_{ij}=g(\nabla e_i,e_j)\,.
\end{equation}
Hence, formulae \eqref{Lawson's covariant derivatives 3} and \eqref{geometric Dirac operator} become
\begin{equation}
\label{Lawson's covariant derivatives 4}
\nabla \varphi=\mathrm{d}\phi+\frac{1}{4} g(\nabla e_i,e_j)\otimes \Delta_n^\mathbb{C}(e_i e_j) \cdot \varphi
\end{equation}
and
\begin{equation}
\label{geometric Levi-Civita Dirac operator}
W\varphi=\Delta_n^\mathbb{C}(e_i)\nabla_{e_i}\varphi\,,
\end{equation}
where we use $W$ to denote the Dirac operator.

Now, we further assume that our manifold is parallelizable. In particular, this assumption is satisfied for any 3-dimensional oriented manifold. Then $P_{\mathrm{SO}}(TX)$ is trivial. Thus, there exists a global section $\mathscr{E}=(e_1,\dots,e_n)$ of $P_{\mathrm{SO}}(TX)$. This implies that formulae \eqref{Lawson's covariant derivatives 4} and \eqref{geometric Levi-Civita Dirac operator} can be extended globally.

\subsection{Definition via frames}
\label{Definition via frames}
In this subsection we set $n=3$. Hence, $N=2$ and $\Delta_3(X):=P_{\mathrm{Spin}}(TX)\times_{\Delta_3^\mathbb{C}} S_2$.
Consider a triple of orthonormal
(with respect to the given metric $g$)
smooth real vector fields $e_j$, $j=1,2,3$.
Each vector $e_j(y)$ has coordinate components $e_j{}^\alpha(y)$, $\alpha=1,2,3$.
The triple of vector fields $e_j$, $j=1,2,3$,
is called an \emph{orthonormal frame}.
We assume that
\begin{equation}
\label{orientation of frame}
\det e_j{}^\alpha>0\,,
\end{equation}
which means that the orientation of our frame agrees with the orientation of our
local coordinates.

Define Pauli matrices
\begin{equation}
\label{Pauli matrices 1}
\sigma^\alpha(y):=
s^j\,e_j{}^\alpha(y)\,,
\end{equation}
where
\begin{equation}
\label{Pauli matrices 2}
s^1:=
\begin{pmatrix}
0&1\\
1&0
\end{pmatrix},
\qquad
s^2:=
\begin{pmatrix}
0&-i\\
i&0
\end{pmatrix},
\qquad
s^3:=
\begin{pmatrix}
1&0\\
0&-1
\end{pmatrix}.
\end{equation}
Note that formula \eqref{Pauli matrices 1} is equivalent to choosing a particular representation of $\Delta_3^\mathbb{C}$ which is given by
\begin{equation}
\label{Pauli representation}
\Delta_3^\mathbb{C}(e_j)=-is_j\,.
\end{equation}
It is not hard to see that this representation is an irreducible representation of $\mathbb{C}l(\mathbb{R}^3)$.

Let
\begin{equation*}
\left\{{{\beta}\atop{\alpha\gamma}}\right\}:=
\frac12g^{\beta\delta}
\left(
\frac{\partial g_{\gamma\delta}}{\partial y^\alpha}
+
\frac{\partial g_{\alpha\delta}}{\partial y^\gamma}
-
\frac{\partial g_{\alpha\gamma}}{\partial y^\delta}
\right)
\end{equation*}
be the Christoffel symbols.

Using formulae \eqref{Levi-Civita 1 form} to \eqref{geometric Levi-Civita Dirac operator},
we conclude that the massless Dirac operator is the
$2\times2$ matrix first order linear differential operator given by
\begin{equation}
\label{definition of Weyl operator}
W:=-i\sigma^\alpha
\left(
\frac\partial{\partial y^\alpha}
+\frac14\sigma_\beta
\left(
\frac{\partial\sigma^\beta}{\partial y^\alpha}
+\left\{{{\beta}\atop{\alpha\gamma}}\right\}\sigma^\gamma
\right)
\right)
\end{equation}
acting on sections of $\Delta_3\,$. Note that the standard basis for the representation \eqref{Pauli representation} is used here and hence $W$ can be thought of as acting on 2-columns of complex-valued scalar fields. See also Remark \ref{Lawson remark 1} and \cite[Appendix A]{jst_part_b}.

\subsection{Analytic definition}
\label{Analytic definition}

Since we are working on a connected oriented 3-manifold, by picking a global section of $P_{\mathrm{SO}}(TX)$
we can regard the operator $W$ as an operator acting on 2-columns of
complex-valued scalar fields. Now we shall extend it to an operator on half-densities.

The \emph{massless Dirac operator on half-densities}, $W_{1/2}\,$,
corresponding to the given metric $g$
is a particular $2\times2$ matrix first order linear differential operator
acting on 2-columns of complex-valued half-densities.
It is defined by the following four conditions:
\begin{gather}
\label{Analytic definition condition 1}
\operatorname{tr}(W_{1/2})_\mathrm{prin}=0\,,
\\
\label{Analytic definition condition 2}
\det(W_{1/2})_\mathrm{prin}(y,p)=-g^{\alpha\beta}(y)\,p_\alpha p_\beta\,,
\\
\label{Analytic definition condition 3}
(W_{1/2})_\mathrm{sub}=
\frac i{16}\,
g_{\alpha\beta}
\{
(W_{1/2})_\mathrm{prin}
,
(W_{1/2})_\mathrm{prin}
,
(W_{1/2})_\mathrm{prin}
\}_{p_\alpha p_\beta},
\\
\label{Analytic definition condition 4}
-i\,\operatorname{tr}
\bigl[
((W_{1/2})_\mathrm{prin})_{p_1}
((W_{1/2})_\mathrm{prin})_{p_2}
((W_{1/2})_\mathrm{prin})_{p_3}
\bigr]
>0\,.
\end{gather}
Here
$y=(y^1,y^2,y^3)$ denotes local coordinates,
$p=(p_1,p_2,p_3)$  denotes the dual variable (momentum),
$(W_{1/2})_\mathrm{prin}(y,p)$ is the principal symbol,
$(W_{1/2})_\mathrm{sub}(y)$ is the subprincipal symbol,
curly brackets denote the generalized Poisson bracket on matrix-functions
\begin{equation*}
\{F,G,H\}:=F_{y^\alpha}GH_{p_\alpha}-F_{p_\alpha}GH_{y^\alpha},
\end{equation*}
and the subscripts $y^\alpha$ and $p_\alpha$
indicate partial derivatives.

The \emph{massless Dirac operator}, $W$, is defined as
\begin{equation}
\label{definition of Weyl operator on half-densities}
W:=
(\det g_{\kappa\lambda})^{-1/4}
\,W_{1/2}\,
(\det g_{\mu\nu})^{1/4}.
\end{equation}
It acts on 2-columns of complex-valued scalar fields.

The analytic definition of the massless Dirac operator
given in this subsection
originates from \cite[Section 8]{jst_review}
and is equivalent to the traditional geometric definition
presented in subsection \ref{Classical geometric definition}.

\subsection{Spin structure}
\label{Spin structure}

The definitions from
subsections \ref{Definition via frames} and \ref{Analytic definition} 
work for any connected oriented Riemannian 3-manifold
and are equivalent. Note, however, that they do not define the massless
Dirac operator uniquely. Namely, let $W$ be a massless
Dirac operator and let $R(y)$ be an arbitrary smooth $2\times2$
special unitary matrix-function \eqref{special unitary matrix-function}.
One can check that then $R^*WR$ is also a massless Dirac operator.

Let us now look at the issue of non-uniqueness of
the massless Dirac operator the other way round.
Suppose that $W$ and $\tilde W$ are two massless Dirac operators.
Does there exist a smooth matrix-function \eqref{special unitary matrix-function}
such that $\tilde W=R^*WR\,$? If the operators $W$ and $\tilde W$
are in a certain sense `close' then the answer is yes, but in general
there are topological obstructions and the answer is no.
This motivates the introduction of the concept of spin structure, see
\cite[Section 7]{jst_review} and \cite{spin} for details.

Fortunately, for the purposes of our paper the issue of spin structure
is irrelevant because it is known \cite[Section 5]{bar_2000},
that the 3-sphere admits a unique spin structure.
In other words, when we work on $\mathbb{S}^3$ equipped with a Riemannian metric $g$
the constructions from
subsections \ref{Analytic definition} and \ref{Definition via frames}
define the massless Dirac operator uniquely modulo the gauge transformation
\eqref{transformation of the Dirac operator under change of frame},
\eqref{special unitary matrix-function}.

\section{Special vector fields on the 3-sphere}
\label{Special vector fields on the 3-sphere}

Working in $\mathbb{R}^4$ and using Cartesian coordinates,
consider the triple of vector fields
$(\mathbf{K}_\pm)_j{}^{\boldsymbol\alpha}$, $j=1,2,3$,
${\boldsymbol\alpha}=1,2,3,4$, defined as
\begin{equation}
\label{Killing vector fields in R4}
\begin{aligned}
(\mathbf{K}_\pm)_1&=
\begin{pmatrix}
-\mathbf{x}^4&\mp\mathbf{x}^3&\pm\mathbf{x}^2&\mathbf{x}^1
\end{pmatrix},
\\
(\mathbf{K}_\pm)_2&=
\begin{pmatrix}
\pm\mathbf{x}^3&-\mathbf{x}^4&\mp\mathbf{x}^1&\mathbf{x}^2
\end{pmatrix},
\\
(\mathbf{K}_\pm)_3&=
\begin{pmatrix}
\mp\mathbf{x}^2&\pm\mathbf{x}^1&-\mathbf{x}^4&\mathbf{x}^3
\end{pmatrix}.
\end{aligned}
\end{equation}

Observe that the vector fields \eqref{Killing vector fields in R4} are tangent to the 3-sphere
\eqref{definition of 3-sphere},
so let us denote by $(K_\pm)_j{}^\alpha$ the restrictions
of the vector fields \eqref{Killing vector fields in R4}
to the 3-sphere. Here the tensor index $\alpha=1,2,3$ corresponds to local coordinates
$y^\alpha$ on $\mathbb{S}^3$.
Note that we have
$\det\{(K_\pm)_j{}^\alpha\}>0$,
which is in agreement with \eqref{orientation of frame}.

\begin{remark}
\label{remarks for Killing vector fields}
The vector fields $(K_\pm)_j\,$, $j=1,2,3$, constructed above
are special because with the standard metric on $\mathbb{S}^3$
they possess the following properties.
\begin{itemize}
\item[(a)]
The vector fields $(K_\pm)_j$ are orthonormal,
\item[(b)]
The vector fields $(K_\pm)_j$ are Killing vector fields.
\item[(c)]
If we write down the Dirac operator $W_\pm$ using $(K_\pm)_j$ as a frame,
then the eigenspinors corresponding to the eigenvalue $\pm\frac32$
are constant spinors.
Of course, for a given operator $W_+$ or $W_-$
one cannot have constant eigenspinors for eigenvalues $+\frac32$
\textbf{and} $-\frac32$ because this  would contradict
the fact that eigenspinors corresponding to different eigenvalues are orthogonal.
\end{itemize}
\end{remark}

Note that the operators $W_+$ and $W_-$
defined in Remark \ref{remarks for Killing vector fields}(c)
are related as
$W_-=R^*W_+R$,
where $R:\mathbb{S}^3\to\mathrm{SU}(2)$ is the restriction of the matrix-function
\[
\pm
\begin{pmatrix}
\mathbf{x}^4+i\mathbf{x}^3&\mathbf{x}^2+i\mathbf{x}^1\\
-\mathbf{x}^2+i\mathbf{x}^1&\mathbf{x}^4-i\mathbf{x}^3
\end{pmatrix}
\]
to the 3-sphere \eqref{definition of 3-sphere}.

The construction of the unperturbed Dirac operator by means of a
triple of orthonormal Killing vector fields
and an immersion of $\mathbb{S}^3$ in $\mathbb{R}^4$
was previously used in \cite{hehl}.

\section{The scalar Laplacian and its pseudoinverse}
\label{The scalar Laplacian and its pseudoinverse}

In this appendix we work on the 3-sphere equipped with standard metric
$(g_0)_{\alpha\beta}(y)$.

Let $f$ be a smooth scalar function on $\mathbb{S}^3$.
Then there exists a unique sequence of homogeneous harmonic polynomials
$p_n(\mathbf{x})$ of degree $n=0,1,2,\ldots$
such that the series
$\sum\limits_{n=0}^{+\infty}p_n(\mathbf{x})$
converges uniformly, together with all its partial derivatives, on the closed unit ball in
$\mathbb{R}^4$, and coincides with $f$ on $\mathbb{S}^3$.

It is known that the eigenvalues of the operator $-\Delta$
acting on $\mathbb{S}^3$ are
$n(n+2)$, $n=0,1,2,\ldots$,
and their multiplicity is
$(n+1)^2$,
which is the dimension of the vector space of
homogeneous harmonic polynomials of degree $n$.
The explicit formula for the action of the operator $(-\Delta)^{-1}$,
the pseudoinverse of $-\Delta$, 
on our function $f$ is
\begin{equation*}
(-\Delta)^{-1}f=
\left.
\sum_{n=1}^{+\infty}\frac{p_n(\mathbf{x})}{n(n+2)}
\right|_{\|\mathbf{x}\|=1}.
\end{equation*}

\section{Comparison with the 3-torus}
\label{Comparison with the 3-torus}

If we leave only the second and fourth terms in the RHS of
\eqref{formula for lambda2 quadratic form},
substitute this expression into
\eqref{formula for lambda2},
drop the subscripts $\pm$
and use \eqref{formula for V0}, 
we get
\begin{equation}
\label{Comparison with the 3-torus equation 1}
\lambda^{(2)}=
-\frac1{16V^{(0)}}\,
\varepsilon_{qks}
\int_{\mathbb{S}^3}
\left(
h_{jq}\left[L_sh_{jk}\right]
+
h_{rq}
\left[(-\Delta)^{-1}L_rL_sL_jh_{jk}\right]
\right)
\rho_0\,dy\,.
\end{equation}
Formula \eqref{Comparison with the 3-torus equation 1}
coincides with the result from
\cite[Theorem 2.1]{torus}
if we put
$V^{(0)}=(2\pi)^3$ (volume of the unperturbed torus), $\rho_0=1$ and
$L_j=\delta_j{}^\alpha\partial_\alpha$,
with $\partial_\alpha$ denoting partial differentiation
in the $\alpha$th cyclic coordinate on the 3-torus.

\section{Eigenvalues for generalized Berger spheres}
\label{Explicit formulae for some eigenvalues of Dirac operator on generalized Berger spheres}

Here we give further explicit expressions for the eigenvalues using the procedure
from Section \ref{Dirac operator on generalized Berger spheres}
where we apply the operator to harmonic polynomials of degree $n$.
For convenience we seek eigenvalues $\mu$ of the operator
$\widetilde{\mathbf{W}}=\mathbf{W}-\nu I$
obtained by dropping the constant term from
\eqref{Dirac for generalized Berger metric Cartesian 1}.

Let  $\kappa=(\kappa_1,\kappa_2,\kappa_3)\in\{\pm1\}^3$, and let 
\[
 N_+:=\{\kappa\in\{\pm1\}^3: \kappa_1\kappa_2\kappa_3=+1\}=\{(1,1,1), (1,-1,-1), (-1,1,-1),(-1,-1,1)\}.
 \]

For each $n\le4$ we give below an explicit formula
for the characteristic polynomial 
$\chi_n(\mu)$ whose roots give the eigenvalues of $\widetilde{\mathbf{W}}$.
For $n\ge5$ formulae become too cumbersome to list, and we do not
have a general formula yet.

\begin{description}
\item[\qquad \bf{Degree }$n=0$.] \[\chi_0(\mu)=\mu^2.\]
\item[\qquad \bf{Degree }$n=1$.] \[\chi_1(\mu)=\prod\limits_{\kappa\in N_+} \left[\mu+\sum_{j=1}^3 \frac{\kappa_j}{a_j}\right]^2.\]
See also formulae
\eqref{lowest negative eigenvalue for generalized Berger metric}--
\eqref{lowest negative eigenvalue for generalized Berger metric extra 3}.
\item[\qquad \bf{Degree }$n=2$.] 
\[
\chi_2(\mu)=\left[\mu^3-\left(4\sum_{j=1}^3a_j^{-2}\right)\mu+\frac{16}{\prod_{j=1}^3 a_j}\right]^6.
\]
\item[\qquad \bf{Degree }$n=3$.] \[
\chi_3(\mu)=\prod_{\kappa\in N_+} \left[\mu^2-\left(\sum_{j=1}^3 \frac{\kappa_j}{a_j}\right)\mu-3\left(\sum_{j=1}^3 a_j^{-2}-\sum_{\substack{j,k=1\\j\ne k}}^3 \frac{\kappa_j\kappa_k}{a_ja_k}\right)\right]^4.
\]
\item[\qquad \bf{Degree }$n=4$.]
\[
\begin{split}
\chi_4(\mu)&=
\left[\mu^5
-\left(20\sum_{j=1}^3a_j^{-2}\right)\mu^3
+\left(\frac{80}{\prod_{j=1}^3 a_j}\right)\mu^2
\vphantom{\left(\sum_{j=1}^3a_j^{-4}+2\sum_{\substack{j,k=1\\j\ne k}}^3 a_j^{-2}c_k^{-2}\right)}
\right.\\
&\left.
+64\left(\sum_{j=1}^3a_j^{-4}+2\sum_{\substack{j,k=1\\j\ne k}}^3 a_j^{-2}a_k^{-2}\right)\mu
-768\,\frac{\sum_{j=1}^3a_j^{-2}}{\prod_{j=1}^3 a_j}\right]^{10}.
\end{split}
\]
\end{description}

\end{appendices}

\end{document}